\documentclass[11pt]{amsart}

\usepackage{amssymb,amscd,amsmath,hyperref,color,enumerate}
\usepackage[all,cmtip]{xy}
\usepackage{mathrsfs}
\usepackage{lipsum}
\usepackage{nicematrix,tikz, ifthen}
\usepackage{scalerel}
\usepackage{chemmacros}
\usepackage{ascii}
\makeatletter
\newenvironment{claimproof}{\par
	\pushQED{\hfill$\diamondsuit$}%
	\normalfont \topsep6\p@\@plus6\p@\relax
	\trivlist
	\item[]\ignorespaces
}{%
	\popQED\endtrivlist\@endpefalse
}
\makeatother

\usepackage[margin=3cm]{geometry}

\newcommand{\N}{{\ensuremath{\mathbb{N}}}}

\newcommand{\C}{{\ensuremath{\mathbb{C}}}}

\usepackage[normalem]{ulem}

\newcommand{\stkout}[1]{\ifmmode\text{\sout{\ensuremath{#1}}}\else\sout{#1}\fi}

\DeclareMathOperator{\supp}{supp}

\DeclareMathOperator{\diag}{diag}

\DeclareMathOperator{\id}{id}

\DeclareMathOperator{\spn}{span}

\newcommand{\Sph}{\ensuremath{\mathbb{S}}}

\newcommand{\ca}[1]{\ensuremath{\mathcal{#1}}}

\newcommand{\abs}[1]{\ensuremath{ {\left| #1 \right|} }}

\newcommand{\norm}[1]{\ensuremath{ {\left\| #1 \right\|} }}

\newtheorem{proposition}{Proposition}[section]
\newtheorem{lemma}[proposition]{Lemma}

\newtheorem{theorem}[proposition]{Theorem}
\newtheorem{claim}{Claim}

\theoremstyle{definition}

\newtheorem{example}[proposition]{Example}

\newtheorem{case}{Case}

\newtheorem{remark}[proposition]{Remark}

\numberwithin{equation}{section}

\newlength{\leftstackrelawd}
\newlength{\leftstackrelbwd}
\def\leftstackrel#1#2{\settowidth{\leftstackrelawd}%
	{${{}^{#1}}$}\settowidth{\leftstackrelbwd}{$#2$}%
	\addtolength{\leftstackrelawd}{-\leftstackrelbwd}%
	\leavevmode\ifthenelse{\lengthtest{\leftstackrelawd>0pt}}%
	{\kern-.5\leftstackrelawd}{}\mathrel{\mathop{#2}\limits^{#1}}}

\newcommand{\tripprox}{\mathrel{\setbox0\hbox{$\approx$}%
		\mbox{\makebox[0pt][l]{\raisebox{0.48\ht0}{$\approx$}}$\approx$}}}

\begin{document}
	
	\title[An extension  of Petek-\v{S}emrl preserver theorems to SMAs]{An extension of Petek-\v{S}emrl preserver theorems for Jordan embeddings of structural matrix algebras}
	
	\author{Ilja Gogi\'{c}, Mateo Toma\v{s}evi\'{c}}
	
	\address{I.~Gogi\'c, Department of Mathematics, Faculty of Science, University of Zagreb, Bijeni\v{c}ka 30, 10000 Zagreb, Croatia}
	\email{ilja@math.hr}
	
	\address{M.~Toma\v{s}evi\'c, Department of Mathematics, Faculty of Science, University of Zagreb, Bijeni\v{c}ka 30, 10000 Zagreb, Croatia}
	\email{mateo.tomasevic@math.hr}
	
	\thanks{We thank Tatjana Petek for her comments on the first  version of the paper. We are also very grateful to the referee for their meticulous review of the paper and for the insightful comments and suggestions they provided.
	}
	

	
	\date{\today}
	
	\maketitle
	\begin{abstract}
		Let $M_n$ be the algebra of $n \times n$ complex matrices and $\mathcal{T}_n \subseteq M_n$ the corresponding upper-triangular subalgebra. In their influential work, Petek and \v{S}emrl characterize Jordan automorphisms of $M_n$ and $\mathcal{T}_n$, when $n \geq 3$, as (injective in the case of $\mathcal{T}_n$)  continuous commutativity and spectrum preserving maps $\phi :  M_n \to M_n$ and $\phi : \mathcal{T}_n \to \mathcal{T}_n$. Recently, in a joint work with Petek, the authors extended this characterization to the
		maps $\phi : \mathcal{A} \to M_n$, where $\mathcal{A}$ is an arbitrary subalgebra of $M_n$ that contains $\mathcal{T}_n$. In particular, any such map $\phi$ is a Jordan embedding and hence of the form  $\phi(X)=TXT^{-1}$ or $\phi(X)=TX^tT^{-1}$, for some  invertible matrix $T\in M_n$. 
		
		In this paper we further extend the aforementioned results in the context of structural matrix algebras (SMAs), i.e.\ subalgebras $\mathcal{A}$ of $M_n$ that contain all diagonal matrices. More precisely, we provide both a necessary and sufficient condition for an SMA $\mathcal{A}\subseteq M_n$ such that any injective continuous commutativity and spectrum preserving map $\phi: \mathcal{A} \to M_n$ is necessarily a Jordan embedding. In contrast to the previous cases, such maps $\phi$ no longer need to be multiplicative/antimultiplicative, nor rank-one preservers.
	\end{abstract}
	
	\section{Introduction}
	A \emph{Jordan homomorphism} between associative algebras (or just rings) $\ca{A}$ and $\ca{B}$  is a linear (additive) map $\phi : \ca{A} \to \ca{B}$ such that 
	$$\phi(ab+ba) = \phi(a)\phi(b) + \phi(b)\phi(a), \qquad \text{ for all }a,b \in \ca{A}.$$ 
	When the algebras (rings) are $2$-torsion-free, this is equivalent to
	$$\phi(a^2) = \phi(a)^2, \qquad \text{ for all }a \in \ca{A}.$$
	The theory of Jordan homomorphisms originates with Jordan algebras, a class of nonassociative algebras (similar to Lie algebras) which appear in various areas, including functional analysis and theoretical quantum mechanics. Since the vast majority of Jordan algebras with relevant applications arise as subalgebras of associative algebras in a natural way, Jordan homomorphisms are often studied precisely in the context of associative algebras. The basic examples of Jordan homomorphisms are algebra homomorphisms and antihomomorphisms. One of the central questions in the Jordan theory is determining conditions on algebras $\ca{A}$ and $\ca{B}$  such that any Jordan homomorphism $\phi : \ca{A} \to \ca{B}$  (possibly with some additional assumptions such as surjectivity)  is multiplicative or antimultiplicative, or more generally, can be written as a suitable combination of such maps. This problem has a rich and long history (see e.g.\ \cite{Benkovic, Herstein, JacobsonRickart, Smiley}).
	
	\smallskip 
	
	Our paper narrows down the scope to matrix algebras. Let $M_n$ be the algebra of $n \times n$ matrices over the field of complex numbers. It is well-known that any nonzero Jordan homomorphism $\phi:M_n\to M_n$ is precisely of the form
	\begin{equation}\label{eq:inner}
		\phi(X) = TXT^{-1} \qquad \text{or}   \qquad \phi(X) = TX^{t}T^{-1},
	\end{equation}
	for some invertible matrix $T \in M_n$ (see e.g.\ \cite{Herstein,Semrl2}). Many attempts were made to characterize Jordan homomorphisms, particularly on matrix algebras, using simple preserving properties. In fact, a basis for this paper is the following (nonlinear) preserver problem which elegantly characterizes Jordan automorphisms of $M_n$:
	\begin{theorem}[{\cite[Theorem~1.1]{Semrl}}]\label{thm:Semrl}
		Let $\phi : M_n \to M_n, n \ge 3$ be a continuous map which preserves commutativity and spectrum. Then there exists an invertible matrix $T \in M_n$ such that $\phi$ is of the form \eqref{eq:inner}.
	\end{theorem}
	
	The first version of this result was actually formulated by Petek and \v Semrl in \cite{PetekSemrl}, where it contained an additional assumption, that $\phi$ preserves rank-one matrices, or commutativity in both directions. Its current form was obtained by \v{S}emrl a decade later in \cite{Semrl}, with the necessity of all assumptions being established via counterexamples. Furthermore, owing to a clever application of the Fundamental theorem of projective geometry, the proof of Theorem \ref{thm:Semrl} turns out to be somewhat shorter than the proof of the initial version, which relied entirely on elementary calculations.
	
	A direct and natural sequel to this result was Petek's paper \cite{Petek} concerning the algebra $\ca{T}_n$ of all upper-triangular matrices in $M_n$. More precisely, Petek gives a complete description of continuous commutativity and spectrum preservers $\ca{T}_n \to \ca{T}_n$. The general form of these maps is somewhat nontrivial, but for $n\geq 3$, by additionally assuming injectivity (or surjectivity), one obtains precisely the maps of the form \eqref{eq:inner}, for a suitable invertible matrix $T \in M_n$. In particular, all such maps are  Jordan automorphisms of $\ca{T}_n$. More recently, in a joint work with Petek, the authors in \cite{GogicPetekTomasevic} obtained a generalization of the aforementioned result to the next natural case: subalgebras $\ca{A} \subseteq M_n$ which contain the upper-triangular algebra $\ca{T}_n$. As it turns out, these are precisely the block upper-triangular algebras (see \cite{WangPanWang}). The result obtained was even more general, showing that any injective continuous commutativity and spectrum preserver $\phi : \ca{A} \to M_n$, where $n \ge 3$ and $\ca{A} \subseteq M_n$ is a block upper-triangular algebra, is in fact a Jordan embedding and therefore of the form \eqref{eq:inner}. 
	
	In order to further extend (and unify) the aforementioned results, the next logical step would be to consider the  subalgebras of $M_n$ which contain all diagonal matrices.  As it turns out, these are precisely unital subalgebras of $M_n$ spanned by some set of matrix units (see \cite[Proposition 3.1]{GogicTomasevic}). Such algebras were originally introduced in the literature by van Wyk in \cite{VanWyk} under the name \emph{structural matrix algebras} (abbreviated as SMAs). They are closely related to incidence algebras and have been studied in many papers, such as \cite{Akkurt, Akkurt2, AkkurtBarkerWild, BeslagaDascalescu, BeslagaDascalescu2, BrusamarelloFornaroliKhrypchenko1, BrusamarelloFornaroliKhrypchenko2, Coelho, Coelho2,GarcesKhrypchenko, GarcesKhrypchenko2, SlowikVanWyk,VanWyk}. In our recent paper \cite{GogicTomasevic}, we determined the general form of Jordan embeddings between two SMAs in $M_n$  (Theorem \ref{thm:Jordan embeddings of SMAs} further in the text), hence extending Coelho's description of their algebra automorphisms \cite[Theorem C]{Coelho}.  In the same paper \cite{GogicTomasevic} we also showed that for an arbitrary SMA $\ca{A} \subseteq M_n$ any linear unital rank-one preserver $\ca{A}\to M_n$ is necessarily a Jordan embedding \cite[Theorem~5.7]{GogicTomasevic}, while the converse fails in general. Furthermore, we described the general form of linear rank preservers $\ca{A} \to M_n$, as maps $X \mapsto S\left(PX + (I-P)X^t\right)T$, for some invertible matrices $S,T \in M_n$ and a central idempotent $P\in\ca{A}$.
	
	\smallskip
	
	The purpose of this paper is to extend the aforementioned Petek-\v Semrl preserver theorems for Jordan embeddings in the context of SMAs. The main result of the paper is Theorem \ref{thm:main result}, which provides both a necessary and sufficient condition on an SMA $\ca{A} \subseteq M_n$  such that any injective continuous commutativity and spectrum preserving map $\mathcal{A} \to M_n$ is necessarily a Jordan embedding. 
	
	\smallskip
	
	This paper is organized as follows. We begin Section \S\ref{sec:preliminaries} by providing relevant terminology, notation and results that will be used throughout the paper. In Section \S\ref{sec:mainresult}, we state and prove our main result, Theorem \ref{thm:main result}. In Section \S\ref{sec:examplesandremarks} we first provide examples which demonstrate that, in contrast to the previous cases of $M_n$, $\ca{T}_n$, and the general block upper-triangular subalgebras of $M_n$, Jordan embeddings of arbitrary SMAs which satisfy the condition (i) of Theorem \ref{thm:main result}, no longer need to be multiplicative/antimultiplicative, nor rank(-one) preservers (Examples \ref{ex:noMAMP} and \ref{ex:SMA-nonlinearOK-nontrivialtransitive}). Finally, we also discuss  the indispensability of all the assumptions of Theorem \ref{thm:main result} (Remark \ref{rem:indispensability}).
	
	\section{Notation and Preliminaries}\label{sec:preliminaries}
	We shall follow the same notation from \cite{GogicTomasevic}. More specifically, 
	for any set $S$, by $\abs{S}$ we denote its cardinality. If $\rho$ is a binary relation on a set $S$, for a fixed $x \in S$ by $\rho(x)$ and $\rho^{-1}(x)$ we denote its image and preimage by $\rho$, respectively, i.e.
	$$\rho(x)= \{y \in S : (x,y) \in \rho\}, \qquad \rho^{-1}(x)= \{y \in S : (y,x) \in \rho\}.$$
	For integers $k \le l$, by $[k,l]$ we denote the set of all integers between $k$ and $l$, inclusive.
	
	\smallskip
	
	Let $n \in \N$.
	\begin{itemize}
		\item We denote by $\Delta_n$ the diagonal relation $\{(i,i) :1 \le i \le n\}$ on $[1,n]$.
		\item As already stated, by $M_n$ we denote the algebra of all $n\times n$ complex matrices. Further, by $\ca{T}_n$ and $\ca{D}_n$ we denote the subalgebras of all upper-triangular and diagonal matrices of $M_n$, respectively. 
		\item For $A,B \in M_n$, by $A \leftrightarrow B$ and $A \perp B$ we denote that  $AB = BA$ and $AB = BA = 0$, respectively.
		\item For $A \in M_n$ by $k_A(x) := \det(x I-A)$ we denote the characteristic polynomial of $A$.
		\item We denote by $\diag(\lambda_1,\ldots,\lambda_n) \in \ca{D}_n$ the diagonal matrix with diagonal entries equal to $\lambda_1,\ldots,\lambda_n\in \C$ (in this order). The similar notation will be used for block diagonal matrices and the corresponding subalgebras.
		\item By $\Lambda_n\in \ca{D}_n$ we denote the diagonal matrix  $\diag(1,\ldots, n)$.
		\item For $i,j \in [1,n]$ we denote by $E_{ij}\in M_n$ the standard matrix unit with $1$ at the position $(i,j)$ and $0$ elsewhere. Similarly, the canonical basis vectors of $\C^n$ are denoted by $e_1, \ldots, e_n$.
		\item For vectors $u,v \in \C^n$ by $u \parallel v$ we denote  that the set $\{u,v\}$ is linearly dependent. The same notation is used for matrices.
		\item For any permutation $\pi \in S_n$ (where, as usual, $S_n$ denotes the symmetric group), by \begin{equation}\label{eq:definition of permutation matrix}
			R_{\pi} := \sum_{k=1}^n E_{k\pi(k)}
		\end{equation} we denote the permutation matrix in $M_n$ associated to $\pi$.
	\end{itemize}
	
	As any matrix $A = [A_{ij}]_{i,j=1}^n \in M_n$ can be understood as a map $[1,n]^2 \to \C, (i,j) \mapsto A_{ij}$, we can consider its \emph{support }$\supp A$ as the set of all indices $(i,j) \in [1,n]^2$ such that $A_{ij} \ne 0$. Moreover, for a set $S \subseteq [1,n]^2$ we also say that the matrix $A$ is \emph{supported in $S$} if $\supp A \subseteq S$. As is customary, we identify vectors in ${\mathbb{C}}^{n}$ with their corresponding column matrices and, in a similar manner, we define their support. 
	
	\smallskip
	
	As usual, given a unital complex algebra $\ca{A}$, by $Z(\ca{A})$ and $\ca{A}^\times$  we denote its centre and the set of all its invertible elements, respectively. By a \emph{quasi-order} on $[1,n]$ we mean a reflexive transitive relation $\rho \subseteq [1,n]^2$. For a quasi-order $\rho$ we define the unital subalgebra of $M_n$ by
	$$\ca{A}_\rho :=\{A \in M_n : \supp A \subseteq \rho\}=\spn\{E_{ij} : (i,j) \in \rho\},$$
	which we call a \emph{structural matrix algebra (SMA) defined by the quasi-order $\rho$}. Throughout the paper we also write 
	$\rho^\times$ for $\rho\setminus \Delta_n$.

	\smallskip
	
	Following \cite{Coelho}, given a quasi-order $\rho$ on $[1,n]$ we say that a map $g : \rho \to \C^\times$ is \emph{transitive} if 
	$$g(i,j)g(j,k) = g(i,k),  \qquad \text{ for all } (i,j), (j,k) \in \rho.$$
	A transitive map $g : \rho \to \C^\times$ is said to be \emph{trivial} if there exists a map $s : [1,n] \to \C^\times$ such that $g$ \emph{separates through $s$}, that is
	$$g(i,j) = \frac{s(i)}{s(j)}, \qquad \text{ for all }(i,j) \in \rho.$$
	Every transitive map $g$ induces an (algebra) automorphism $g^*$ of $\ca{A}_\rho$, defined on the basis of  matrix units as
	\begin{equation}\label{eq:inducedauto}
		g^*(E_{ij}) = g(i,j)E_{ij}, \qquad \text{ for all }(i,j) \in\rho.
	\end{equation}
	We explicitly state the results from our recent papers \cite{GogicPetekTomasevic} and \cite{GogicTomasevic}, which will be essentially used later on a few occasions.
	
	\begin{theorem}[{\cite[Theorem 1.4]{GogicPetekTomasevic}}]\label{thm:parabolic}
		Let $\ca{A} \subseteq M_n$ be a block upper-triangular subalgebra. Then every injective continuous commutativity and spectrum preserver $\phi : \ca{A} \to M_n$ is of the form \eqref{eq:inner} for some $T \in M_n^\times$, and in particular a Jordan embedding.
	\end{theorem}
	
	\begin{theorem}[{\cite[Theorem 4.9]{GogicTomasevic}}]\label{thm:Jordan embeddings of SMAs}
		Let $\ca{A}_\rho \subseteq M_n$ be an SMA and let $\phi : \ca{A}_\rho \to M_n$ be a Jordan embedding. Then there exists an invertible matrix $S \in M_n^\times$, a central idempotent $P \in \ca{A}_\rho$, and a transitive map $g : \rho \to \C^\times$ such that
		$$\phi(X) = S(Pg^*(X) + (I-P)g^*(X)^t)S^{-1}, \qquad \mbox{ for all } X \in \ca{A}_\rho.$$
	\end{theorem}
	
	\begin{theorem}[{\cite[Theorem 3.4]{GogicTomasevic}}]\label{thm:inner diagonalization on SMA}
		Let $\ca{A}_\rho \subseteq M_n$ be an SMA and let $\ca{F} \subseteq \ca{A}_\rho$ be a commuting family of diagonalizable matrices. Then there exists $S \in \ca{A}_\rho^\times$ such that $S^{-1}\ca{F}S \subseteq \ca{D}_n$.
	\end{theorem}
	
	\smallskip 
	
	At the end of this preliminary section, we also introduce the following auxiliary notation, which will play a crucial role in the proof of Theorem \ref{thm:main result}. Let $A \in M_n$ and $S \subseteq [1,n]$. \begin{itemize}
		\item If $S \ne [1,n]$, denote by $A^{\flat S} \in M_{n - \abs{S}}$ the matrix obtained from $A$ by deleting all rows $i$ and columns $j$ where $i,j \in S$. We also formally allow $A^{\flat \emptyset} = A$.
		\item Denote by $A^{\sharp S} \in M_{n + \abs{S}}$ the matrix obtained from $A$ by adding zero rows and columns so that $(A^{\sharp S})^{\flat S} = A$.
	\end{itemize}

	Note that if $S = \{s_1,\ldots,s_k\}$, where $s_1 < \cdots < s_k$, we have
	\begin{equation}
		\label{eq:decomposition of flat and sharp}
		A^{\sharp S} = (\cdots (A^{\sharp \{s_1\}})^{\sharp \{s_{2}\}}\cdots)^{\sharp \{s_k\}}, \qquad A^{\flat S} = (\cdots (A^{\flat \{s_k\}})^{\flat \{s_{k-1}\}}\cdots)^{\flat \{s_1\}}.
	\end{equation}
	
	The next simple lemmas outline some key properties of the maps $(\cdot)^{\flat S}$ and $(\cdot)^{\sharp S}$.
	
	\begin{lemma}\label{le:sharp is homomorphism}
		For each $S \subseteq [1,n]$, the map $(\cdot)^{\sharp S} : M_n \to M_{n + \abs{S}}$ is an algebra monomorphism.
	\end{lemma}
	\begin{proof}
		Obviously, $(\cdot)^{\sharp S}$ is a linear map. In view of \eqref{eq:decomposition of flat and sharp}, it suffices to assume that $S = \{k\}$ for some $k \in [1,n]$. Moreover, we can further assume that $k \notin \{1,n\}$ as the proof for the cases $k \in \{1,n\}$ is conceptually similar, only even simpler in terms of notation. By using block-matrix notation, if we denote $X,Y \in M_n$ as
		$$X = \begin{bmatrix}
			X^{I,I}_{(k-1)\times (k-1)} & X^{I,II}_{(k-1)\times(n-k+1)}\\
			X^{II,I}_{(n-k+1)\times (k-1)} & X^{II,II}_{(n-k+1)\times(n-k+1)}
		\end{bmatrix}, \qquad Y = \begin{bmatrix}
			Y^{I,I}_{(k-1)\times (k-1)} & Y^{I,II}_{(k-1)\times(n-k+1)}\\
			Y^{II,I}_{(n-k+1)\times (k-1)} & Y^{II,II}_{(n-k+1)\times(n-k+1)}
		\end{bmatrix}$$
		we then have
		$$X^{\sharp \{k\}} = \begin{bmatrix}
			X^{I,I}_{(k-1)\times (k-1)} & 0_{(k-1)\times 1} & X^{I,II}_{(k-1)\times(n-k+1)}\\
			0_{1 \times (k-1)} & 0_{1\times 1} & 0_{1 \times (n-k+1)} \\
			X^{II,I}_{(n-k+1)\times (k-1)} & 0_{(n-k+1) \times 1} & X^{II,II}_{(n-k+1)\times(n-k+1)}
		\end{bmatrix}$$
		and $$Y^{\sharp \{k\}} = \begin{bmatrix}
			Y^{I,I}_{(k-1)\times (k-1)} & 0_{(k-1)\times 1} & Y^{I,II}_{(k-1)\times(n-k+1)}\\
			0_{1 \times (k-1)} & 0_{1\times 1} & 0_{1 \times (n-k+1)} \\
			Y^{II,I}_{(n-k+1)\times (k-1)} & 0_{(n-k+1) \times 1} & Y^{II,II}_{(n-k+1)\times(n-k+1)}
		\end{bmatrix},$$
		where both matrices are in $M_{n+1}$. By using block-matrix multiplication (see e.g.\ \cite[Section 2]{Stewart}), it follows that $X^{\sharp \{k\}} Y^{\sharp \{k\}}$ equals
		\begin{small}
			\begin{align*}
				&\begin{bmatrix}
					(X^{I,I}Y^{I,I}+X^{I,II}Y^{II,I})_{(k-1)\times (k-1)} & 0_{(k-1)\times 1} & (X^{I,I}Y^{I,II}+X^{I,II}Y^{II,II})_{(k-1)\times(n-k+1)}\\
					0_{1 \times (k-1)} & 0_{1\times 1} & 0_{1 \times (n-k+1)} \\
					(X^{II,I}Y^{I,I}+X^{II,II}Y^{II,I})_{(n-k+1)\times (k-1)} & 0_{(n-k+1) \times 1} & (X^{II,I}Y^{I,II} + X^{II,II}Y^{II,II})_{(n-k+1)\times(n-k+1)}\\
				\end{bmatrix}\\
				&\quad= \begin{bmatrix}
					(X^{I,I}Y^{I,I}+X^{I,II}Y^{II,I})_{(k-1)\times (k-1)}  & (X^{I,I}Y^{I,II}+X^{I,II}Y^{II,II})_{(k-1)\times(n-k+1)}\\
					(X^{II,I}Y^{I,I}+X^{II,II}Y^{II,I})_{(n-k+1)\times (k-1)} & (X^{II,I}Y^{I,II} + X^{II,II}Y^{II,II})_{(n-k+1)\times(n-k+1)}
				\end{bmatrix}^{\sharp \{k\}},
			\end{align*}
		\end{small}
		while the latter matrix is precisely $(XY)^{\sharp \{k\}}$. Finally, the injectivity of $(\cdot)^{\sharp \{k\}}$ is clear.
	\end{proof}
	
	\begin{lemma}\label{le:flat is an isomorphism}
		For a nonempty $S\subseteq [1,n]$, the set \begin{equation}\label{eq:M_n subset}
			M_n^{\subseteq S} := \{X \in M_n : \supp X \subseteq S\times S\}
		\end{equation}
		is a subalgebra of $M_n$ and 
		$$(\cdot)^{\flat S^c} : M_n^{\subseteq S} \to M_{\abs{S}}$$
		is an algebra isomorphism.
	\end{lemma}
	\begin{proof}
		Notice first that $(X^{\flat S^c})^{\sharp S^c} = X$ for each $X \in M_n^{\subseteq S}$, hence $(\cdot)^{\flat S^c}$ is injective. It is obvious that $M_n^{\subseteq S}$ is a subspace and that $(\cdot)^{\flat S}$ is a linear map. For $X,Y \in M_n^{\subseteq S}$ and $(i,j) \in [1,n]^2 \setminus (S \times S)$, for each $k \in [1,n]$ we have $(i,k) \notin S \times S$ or $(k,j) \notin S \times S$ and hence $(XY)_{ij} = \sum_{k=1}^n X_{ik}Y_{kj} = 0$. We conclude that $\supp (XY) \subseteq S \times S$, so $XY \in M_n^{\subseteq S}$. Moreover, we have
		$$((XY)^{\flat S^c})^{\sharp S^c} = XY = (X^{\flat S^c})^{\sharp S^c} (Y^{\flat S^c})^{\sharp S^c} \stackrel{\text{Lemma } \ref{le:sharp is homomorphism}}= (X^{\flat S^c}Y^{\flat S^c})^{\sharp S^c}.$$
		The injectivity of $(\cdot)^{\sharp S^c}$ from Lemma \ref{le:sharp is homomorphism} yields
		$$(XY)^{\flat S^c} = X^{\flat S^c}Y^{\flat S^c},$$
		which shows that $(\cdot)^{\flat S^c}$ is also a multiplicative map and thus an algebra homomorphism. As $\dim M_n^{\subseteq S} = \dim M_{\abs{S}} = \abs{S}^2$, the map is an isomorphism.
	\end{proof}
	We also extend the notation $(\cdot)^{\flat S}$ and $(\cdot)^{\sharp S}$ notation to sets of matrices by applying the respective operation elementwise.
	
	\begin{lemma}\label{le:deletion of SMA is SMA}
		Let $\ca{A}_{\rho} \subseteq M_n$ be an SMA. For each nonempty $S \subseteq [1,n]$, $\ca{A}_{\rho}^{\flat S}$ is again an SMA in $M_{n-\abs{S}}$.
	\end{lemma}
	\begin{proof}
		In view of \eqref{eq:decomposition of flat and sharp}, it suffices to assume that $S = \{k\}$ for some $k \in [1,n]$. Consider the bijection
		$$\kappa : [1,n]\setminus \{k\} \to [1,n-1], \qquad \kappa(j) := \begin{cases}
			j, &\text{ if }j < k,\\
			j-1, &\text{ if }j > k
		\end{cases}$$
		and the relation
		$$\rho' := \{(\kappa(i),\kappa(j)) : (i,j) \in \rho \text{ and } i,j \ne k\} \subseteq [1,n-1]^2.$$
		We claim that $\rho'$ is a quasi-order on $[1,n-1]$. Let $j \in [1,n-1]$ be arbitrary. We have
		$$(j,j) = (\kappa(\kappa^{-1}(j)), \kappa(\kappa^{-1}(j))).$$
		Since $(\kappa^{-1}(j), \kappa^{-1}(j)) \in \rho$ and $\kappa^{-1}(j) \ne k$, it follows that $(j,j) \in \rho'$ and thus $\rho'$ is reflexive. To show the transitivity of $\rho'$, consider $i,j,l \in [1,n] \setminus \{k\}$ such that $(\kappa(i),\kappa(j)), (\kappa(j),\kappa(l)) \in \rho'$. This means that $(i,j),(j,l)\in\rho$ which, by the transitivity of $\rho$, implies that $(i,l) \in \rho$ and hence $(\kappa(i),\kappa(l)) \in \rho'$. Finally, from the definition of $(\cdot)^{\flat \{k\}}$ and $\rho'$ it is clear that $\ca{A}_{\rho}^{\flat \{k\}} = \ca{A}_{\rho'}$, and the latter is an SMA.
	\end{proof}
	
	\section{Main result}\label{sec:mainresult}
	As already announced, in this section we prove our main result, Theorem \ref{thm:main result}, which is an extension of Petek-\v Semrl theorems to SMAs. Since the rank-one matrices play an essential role in the proof of  Theorem \ref{thm:main result},  our first task is to determine the (norm-)closure of rank-one diagonalizable matrices (i.e.\ non-nilpotents) in SMAs. Throughout this section, given an SMA $\ca{A}_\rho \subseteq M_n$,  we shall use the following notation
	$$
	\ca{R} := \{A \in \ca{A}_\rho : A \text{ is a rank-one non-nilpotent}\}
	$$
	and by $\overline{\ca{R}}$ we denote its closure.
	\begin{lemma}\label{le:closure of R}
		Let $\ca{A}_\rho \subseteq M_n$ be an SMA. Then  $\overline{\ca{R}}$  is given by
		$$\overline{\ca{R}} = \{ab^* : a,b \in \C^n, \exists k \in [1,n] \text{ such that } ae_k^*, e_kb^* \in \ca{A}_\rho\} \subseteq \ca{A}_\rho.$$
		In particular, $\overline{\ca{R}}$ contains all matrices in $\ca{A}_\rho$ supported in a single row or a single column. 
	\end{lemma}
	\begin{proof}
		Note that
		$$\ca{R} = \{uv^* \in \ca{A}_\rho : u,v \in \C^n, v^*u \ne 0\}.$$
		\fbox{$\subseteq$} By the lower semicontinuity of the rank, clearly any nonzero element $A \in \overline{\ca{R}}$ has rank one and hence is of the form $A = ab^*$ for some nonzero vectors $a,b \in \C^n$. Since $A \in \ca{A}_\rho$, we have
		$$(\supp a) \times (\supp b) = \supp A \subseteq \rho.$$
		For the sake of concreteness, we assume that $M_n$ is equipped with the norm 
		$$
		\|X\|_\infty:=\max_{1\leq i,j \leq n} |X_{ij}|.
		$$
		Denote 
		$$
		\mu := \min_{(i,j) \in \supp A} \abs{A_{ij}} > 
		0.
		$$
		By the assumption, there exists a matrix $uv^* \in \ca{R}$ (where $u,v \in \C^n, v^*u \ne 0$) such that $\norm{uv^* - A}_\infty < \mu.$ In particular, for each $(i,j) \in \supp A$ we have
		$$\abs{A_{ij}}-\abs{(uv^*)_{ij}} \le \abs{A_{ij}-(uv^*)_{ij}} < \mu \implies \abs{(uv^*)_{ij}} > \abs{A_{ij}} - \mu \ge 0$$
		so $u_i\overline{v}_j = (uv^{*})_{ij} \ne 0$. It follows
		$$(\supp a) \times (\supp b) = \supp A \subseteq \supp(uv^*) = (\supp u) \times (\supp v),$$
		which implies $\supp a\subseteq \supp u$ and $\supp b \subseteq \supp v.$
		Since $v^*u \ne 0$, we can choose some $1 \le k \le n$ such that $u_k\overline{v}_k \ne 0$. Then $k \in (\supp u) \cap (\supp v)$ and therefore
		\begin{align*}
			((\supp a) \times \{k\}) \cup (\{k\} \times (\supp b)) & \subseteq ((\supp u) \times \{k\}) \cup (\{k\} \times (\supp v)) \\
			& \subseteq (\supp u) \times (\supp v) \subseteq \rho.
		\end{align*}
		In particular, we have
		$$\supp(ae_k^*) = (\supp a) \times \{k\} \subseteq \rho \implies ae_k^* \in \ca{A}_\rho$$
		and
		$$\supp(e_kb^*) = \{k\} \times (\supp b) \subseteq \rho \implies e_kb^* \in \ca{A}_\rho.$$
		\fbox{$\supseteq$} Clearly $0 \in \overline{\ca{R}}$. Suppose now that $ab^* \in \ca{A}_\rho$ for some nonzero vectors $a,b \in \C^n$ such that $ae_k^*, e_kb^* \in \ca{A}_\rho$ for some $1 \le k \le n$. We claim that $ab^* \in \overline{\ca{R}}$. If $b^*a \ne 0$, then clearly $ab^* \in \ca{R} \subseteq \overline{\ca{R}}$, so assume $b^*a = 0$.\smallskip
		
		\noindent For $\varepsilon > 0$ consider
		$$A_\varepsilon := (a+\varepsilon e_k)(b+\varepsilon e_k)^* \in M_n.$$
		We have
		$$A_\varepsilon = ab^* + \varepsilon (ae_k^* + e_kb^*) + \varepsilon^2 E_{kk} \in \ca{A}_\rho$$
		and clearly $\lim_{\varepsilon \to 0} A_\varepsilon = ab^*$. Furthermore,
		$$(b+\varepsilon e_k)^*(a+\varepsilon e_k) = b^*a + \varepsilon (e_k^* a + b^*e_k) + \varepsilon^2 e_ke_k^* = \varepsilon (a_k + \overline{b_k}) + \varepsilon^2,$$
		which is nonzero when $\varepsilon \ne -(a_k + \overline{b_k})$, implying that $A_\varepsilon \in \ca{R}$ for such $\varepsilon$. This completes the proof of the inclusion.
		
		\smallskip
		
		Finally, suppose that a matrix $A \in \ca{A}_{\rho}$ is supported in a single row $j \in [1,n]$. Then there exists a vector $b \in \C^n$ such that $A = e_j b^*$. We have $e_j b^*, e_je_j^* = E_{jj} \in \ca{A}_{\rho}$ so $A \in \overline{\ca{R}}$. The case when a matrix is supported in a single column is treated by a similar argument.
	\end{proof}
	
	The next example shows that for a general SMA $\ca{A}_\rho \subseteq M_n$, the set $\ca{R}$ does not need to be dense in the set of all rank-one matrices in $\ca{A}_\rho$.
	\begin{example}\label{ex:Rnotclosed}
		Consider the quasi-order
		$$
		\rho:= \Delta_4 \cup \{(1,3),(1,4),(2,3),(2,4)\}$$ 
		on $[1,4]$ and the corresponding SMA 
		$$
		A_\rho=\begin{bmatrix}
			* & 0 & * & * \\
			0 & * & * & * \\
			0 & 0 & * & 0 \\
			0 & 0 & 0 & *
		\end{bmatrix} \subseteq \ca{T}_4.
		$$
		Then
		$$
		A:= \begin{bmatrix}
			0 & 0 & 1 & 1 \\
			0 & 0 & 1 & 1 \\
			0 & 0 & 0 & 0 \\
			0 & 0 & 0 & 0
		\end{bmatrix}\in \ca{A}_\rho
		$$
		is a rank-one matrix such that $A \notin\overline{\ca{R}}$, as for any $0<\varepsilon<1$, the $\|\cdot\|_{\infty}$-ball $B(A,\varepsilon)$ in $\ca{A}_\rho$ does not intersect $\ca{R}$. Indeed, if $X \in B(A,\varepsilon)$ is a diagonalizable matrix in $\ca{A}_\rho$, then there exists some $j \in [1,4]$ such that $([1,2]\times [3,4]) \cup \{(j,j)\} \subseteq \supp X$. Hence, $X$ cannot be of rank-one. 
		
		
		\smallskip
		
		Alternatively, suppose that $A = ab^*$ for some nonzero vectors $a,b \in \C^n$. Then one easily sees that $a \parallel (e_1+e_2)$ and $b \parallel (e_3+e_4)$. On the other hand, we have
		$$(e_1+e_2)e_1^*, (e_1+e_2)e_2^*, e_3(e_3+e_4)^*, e_4(e_3+e_4)^* \notin \ca{A}_\rho$$
		and these conclusions are invariant under scalar multiplication by a nonzero scalar. It follows that $A$ does not satisfy the condition of Lemma \ref{le:closure of R} and thus $A \notin\overline{\ca{R}}$.   
	\end{example}
	\begin{remark}
		In fact, given a quasi-order $\rho$ on $[1,n]$, one easily sees that $\ca{R}$ is dense in the set of all rank-one matrices in $\ca{A}_\rho$ if and only if for all subsets $S,T \subseteq [1,n]$ we have
		$$S \times T \subseteq \rho \implies \exists k \in [1,n] \text{ such that } (S\times\{k\})\cup (\{k\}\times T) \subseteq \rho.$$
		It is not difficult to check that this condition is fulfilled for all block upper-triangular subalgebras of $M_n$. This is reflected in the fact that the maps resulting from Theorem \ref{thm:parabolic} are automatically rank-one (and, a posteriori, rank) preservers (see Claim \ref{cl:preserves rank-one} further in the text and compare it to Step 5.2 in the proof of  \cite[Theorem~1.4]{GogicPetekTomasevic}).
	\end{remark}
	\begin{lemma}\label{le:diagonals are dense}
		Let $\ca{A}_\rho \subseteq M_n$ be an SMA. Then the set
		$$\{S\diag(\lambda_1,\ldots,\lambda_n)S^{-1} : S\in \ca{A}_\rho^\times, \lambda_1,\ldots,\lambda_n\in \C \text{ pairwise distinct}\}$$
		is dense in $\ca{A}_\rho$.
	\end{lemma}
	\begin{proof} \setcounter{case}{0}
		Following the notation of \cite{GogicPetekTomasevic}, by  $\ca{A}_{k_1,\ldots,k_p}$ 
		we denote the block upper-triangular subalgebra of $M_n$ whose diagonal blocks are $M_{k_1}, \ldots, M_{k_p}$, where $k_1+\cdots + k_p=n$. 
		\begin{case} First we consider the case when
			$\diag(M_{k_1}, \ldots, M_{k_p}) \subseteq \ca{A}_\rho \subseteq \ca{A}_{k_1,\ldots,k_p}$. Let $A \in \ca{A}_\rho$ be arbitrary and let $\varepsilon > 0$. By applying the Schur triangularization on each diagonal block, we obtain a unitary block-diagonal matrix $U \in \ca{A}_\rho^\times$ such that $U^*AU \in \ca{T}_n$. Let $\Theta \in \ca{T}_n$ be a matrix which is identical to $U^*AU$ outside the diagonal, while its diagonal $D \in \ca{D}_n$ consists of pairwise distinct complex numbers $\Theta_{11}, \ldots, \Theta_{nn}$ such that
			$$\sum_{k \in [1,n]} \abs{(U^*AU)_{kk} - \Theta_{kk}}^2 < \varepsilon^2.$$
			Clearly, $\supp \Theta \subseteq \supp (U^*AU) \cup \Delta_n \subseteq \rho$, so $\Theta \in \ca{A}_\rho$. If $\norm{\cdot}_F$ denotes the Frobenius norm on $M_n$, we have
			$$\norm{A - U\Theta U^*}_F = \norm{U^*AU - \Theta}_F = \left(\sum_{k \in [1,n]} \abs{(U^*AU)_{kk} - \Theta_{kk}}^2\right)^{\frac12} < \varepsilon.$$
			Since $U\Theta U^* \in \ca{A}_{\rho}$ has $n$ distinct eigenvalues, it remains to apply Theorem \ref{thm:inner diagonalization on SMA}.
		\end{case}
		
		\begin{case}
			Now we consider the general case. By  \cite[p.\ 432]{Akkurt} (see also \cite[Lemma 3.2]{GogicTomasevic}), there exists a permutation $\pi \in S_n$ such that $$\diag(M_{k_1}, \ldots, M_{k_p}) \subseteq R_{\pi}\ca{A}_{\rho} R_{\pi}^{-1} \subseteq \ca{A}_{k_1,\ldots,k_p},$$
			where $R_\pi\in M_n^\times$ is defined by \eqref{eq:definition of permutation matrix}. By Case 1, the set $$\ca{S} := \{S\diag(\lambda_1,\ldots,\lambda_n)S^{-1} : S\in (R_{\pi}\ca{A}_{\rho} R_{\pi}^{-1})^\times =R_{\pi}\ca{A}_{\rho}^{\times} R_{\pi}^{-1}, \lambda_1,\ldots,\lambda_n \in \C\text{ p.\ d.}\}$$
			is dense in $R_{\pi}\ca{A}_{\rho} R_{\pi}^{-1}$, which immediately implies that $R_{\pi}^{-1}\ca{S}R_{\pi}$, which equals 
			\begin{align*}
				&\{(R_{\pi}^{-1}SR_{\pi})\diag(\lambda_{\pi^{-1}(1)},\ldots,\lambda_{\pi^{-1}(n)})(R_{\pi}^{-1}SR_{\pi})^{-1} : S\in R_{\pi}\ca{A}_{\rho}^{\times} R_{\pi}^{-1}, \lambda_1,\ldots,\lambda_n \in \C\text{ p.\ d.}\}\\
				&= \{T\diag(\mu_1,\ldots,\mu_n)T^{-1} : T\in \ca{A}_\rho^\times, \mu_1,\ldots,\mu_n \in \C\text{ p.\ d.}\},
			\end{align*}
			is dense in $\ca{A}_{\rho}$ (where p.\ d.\ abbreviates ``pairwise distinct'').
		\end{case}
	\end{proof}
	
	Following the notation from \cite{GogicTomasevic}, for a quasi-order $\rho$ on $[1,n]$, by $\tripprox_0$ we denote the relation on $[1,n]$ given by
	$$i \tripprox_0 j \stackrel{\text{def}}\iff (i,j) \in \rho \text{ or }(j,i) \in \rho,$$
	while by $\tripprox$ we denote its transitive closure, which is an equivalence relation. The respective quotient set $[1,n]/\mathop{\tripprox}$ will be denoted by $\ca{Q}$. Now we introduce a new auxiliary definition: we say that an SMA $\ca{A}_\rho \subseteq M_n$ is \emph{2-free} if $\abs{C} \ne 2$ for all $C \in \ca{Q}$. Note that by \cite[Remark 3.3]{GogicTomasevic}, each SMA $\ca{A}_\rho \subseteq M_n$ is isomorphic to a direct sum of central SMAs (each contained in some $M_k$ for some $1 \le k \le n$). In this context, an SMA is $2$-free if and only if this decomposition does not possess a central summand contained in $M_2$ (i.e.\ isomorphic either to $M_2$ or $\ca{T}_2$). Also note that this condition precludes the existence of nonlinear maps $\ca{A} \to M_n$ such as on $M_2$ (from \cite[Example 7]{PetekSemrl}) and $\ca{T}_2$ (from \cite[Theorem 4]{Petek}).
	
	\begin{remark}\label{re:invariant under similarity}
		Suppose that $\ca{A} \subseteq M_n$ is an arbitrary subalgebra and let $S \in M_n^\times$. Clearly, every injective continuous commutativity and spectrum preserver $\ca{A} \to M_n$ is a Jordan embedding if and only if the same holds for the corresponding maps $S^{-1}\ca{A}S \to M_n$.
	\end{remark}
	
	\begin{proposition}\label{prop:necessarily multiplicative}
		Let $\ca{A}_\rho \subseteq M_n$ be a 2-free SMA and let $\phi : \ca{A}_\rho \to M_n$ be an injective continuous commutativity and spectrum preserver. Then there exists $S \in M_n^\times$, a transitive map $g : \rho \to \C^\times$ and quasi-orders $\rho_M^{\phi}, \rho_A^{\phi} \subseteq \rho$ such that  $\rho^\phi_M \cup \rho^\phi_A = \rho, \rho^\phi_M \cap \rho^\phi_A = \Delta_n$ and
		$$\phi(E_{ij}) = \begin{cases}
			g(i,j)SE_{ij} S^{-1}, &\text{ if } (i,j) \in \rho_M^{\phi},\\
			g(i,j)SE_{ji} S^{-1}, &\text{ if } (i,j) \in \rho_A^{\phi}.
		\end{cases}$$
	\end{proposition}
	\begin{proof}
		In view of Theorem \ref{thm:Semrl}, we may assume throughout that $\ca{A}_\rho \subsetneq M_n$.
		\begin{claim}\label{cl:preserves characteristic polynomial}
			$\phi$ preserves characteristic polynomial.
		\end{claim}
		\begin{claimproof}
			$\phi$ clearly preserves characteristic polynomial on the set $$\{S\diag(\lambda_1,\ldots,\lambda_n)S^{-1} : S\in \ca{A}_\rho^\times, \lambda_1,\ldots,\lambda_n \in \C\text{ pairwise distinct}\},$$
			so the claim follows by the continuity of $\phi$ and the characteristic polynomial $k_{(\cdot)} : M_n \to \C[x]$ (see also the proof of \cite[Lemma 4.1]{GogicPetekTomasevic}).
		\end{claimproof}
		
		\begin{claim}\label{cl:identity on diagonals}
			Without loss of generality we can assume $\phi(\Lambda_n) = \Lambda_n$ and hence that $\phi$ acts as the identity map on $\mathcal{D}_n$.
		\end{claim}
		\begin{claimproof}
			Since the matrix $\phi(\Lambda_n) \in M_n$ has eigenvalues $1,\ldots, n$, there exists an $S \in M_n^{\times}$ such that $\phi(\Lambda_n) = S\Lambda_n S^{-1}$. By passing to the map $S^{-1}\phi(\cdot)S$, we can assume $\phi(\Lambda_n) = \Lambda_n.$ Fix an arbitrary $D \in \mathcal{D}_n$. We have $D \leftrightarrow \Lambda_n$ and hence $\phi(D) \leftrightarrow \phi(\Lambda_n) = \Lambda_n$. We conclude $\phi(D) \in \ca{D}_n$. The same argument from \cite[Lemma~2.1]{Semrl} (see also the proof of Step 1 of \cite[Theorem~1.4]{GogicPetekTomasevic}) now gives that $\phi(D)=D$. 
			
		\end{claimproof}
		
		In view of Claim \ref{cl:identity on diagonals}, in the sequel we assume that $n \ge 3$ (as when $n < 3$, the only 2-free SMA is $\ca{D}_n$) and that $\phi$ acts as the identity map on $\mathcal{D}_n$.
		
		\begin{claim}\label{cl:preserves diagonalizability}
			For each $S \in \mathcal{A}_\rho^\times$ there exists $T \in M_n^\times$ such that $$\phi(SDS^{-1}) = TDT^{-1}, \qquad \text{ for all }D \in \mathcal{D}_n.$$
		\end{claim}
		\begin{claimproof}
			For a fixed $S \in \ca{A}_\rho^\times$ there exists $T \in M_n^\times$ such that $\phi(S\Lambda_n S^{-1}) = T \Lambda_n T^{-1}$. Now we can apply Claim \ref{cl:identity on diagonals} to the map $T^{-1}\phi(S(\cdot)S^{-1})T$ which satisfies the same properties as $\phi$, as well as $\Lambda_n \mapsto \Lambda_n$.
		\end{claimproof}
		
		\begin{claim}\label{cl:preserves zero-product}
			Let $A,B \in \mathcal{A}_\rho$ be two diagonalizable matrices such that $A \perp B$. Then $\phi(A) \perp \phi(B)$.
		\end{claim}
		\begin{claimproof}
			Follows directly from Theorem \ref{thm:inner diagonalization on SMA} and Claim \ref{cl:preserves diagonalizability}.
		\end{claimproof}
		
		\begin{claim}\label{cl:phi is homogeneous}\phantom{x}
			\begin{enumerate}[(a)]
				\item Let $A,B \in \ca{A}_\rho$ be diagonalizable matrices such that $A \leftrightarrow B$. Then $\phi(\alpha A + \beta B) = \alpha \phi(A) + \beta\phi(B)$ for all $\alpha,\beta \in \C$.
				\item $\phi$ is a homogeneous map.
			\end{enumerate}
		\end{claim}
		\begin{claimproof}
			(a) follows directly from Theorem \ref{thm:inner diagonalization on SMA} and Claim \ref{cl:preserves diagonalizability}, while (b) follows from (a), Lemma \ref{le:diagonals are dense} and the continuity of $\phi$.
		\end{claimproof}
		
		\begin{claim}\label{cl:preserves rank-one}
			$\phi$ maps every nonzero matrix from $\overline{\ca{R}}$ to a rank-one matrix.
		\end{claim}
		\begin{claimproof}
			If $A \in \ca{R}$, then $A$ is diagonalizable in $\ca{A}_\rho$ (Theorem \ref{thm:inner diagonalization on SMA}) so the assertion follows directly from Claim \ref{cl:preserves diagonalizability}.
			
			Now suppose that $A \in \overline{\ca{R}}$ and let $(A_k)_{k\in \N}$ be a sequence of matrices in $\ca{R}$ such that $A_k \to A$. By continuity we have $\phi(A_k) \to \phi(A)$ and then by lower semicontinuity of the rank we conclude that $\phi(A)$ has rank one.
		\end{claimproof}
		
		\begin{claim}\label{cl:preserves orthogonality}
			Suppose that nonzero matrices $A_1,A_2 \in \overline{\ca{R}}$ satisfy $A_1 \perp A_2$. Then $\phi(A_1) \perp \phi(A_2)$.
		\end{claim}
		\begin{claimproof}
			In view of Claim \ref{cl:preserves rank-one}, for $j=1,2$ denote $\phi(A_j) = x_jy_j^*$ for some nonzero vectors $x_j,y_j \in \C^n$. Since in particular $A_1 \leftrightarrow A_2$, we obtain
			$$(y_1^*x_2) x_1y_2^* = (x_1y_1^*)(x_2y_2^*) = (x_2y_2^*)(x_1y_1^*) = (y_2^*x_1) x_2y_1^*.$$
			If $y_1^*x_2 = y_2^*x_1 = 0$, it follows $\phi(A_1) \perp \phi(A_2)$, as desired. Assume therefore $y_1^*x_2, y_2^*x_1 \ne 0$. Then $x_1y_2^* \parallel x_2y_1^*$, so $x_1 \parallel x_2$ and $y_1 \parallel y_2$. It follows $\phi(A_1) = x_1y_1^* \parallel x_2y_2^* = \phi(A_2)$, so by the injectivity and the homogeneity (Claim \ref{cl:phi is homogeneous}) of $\phi$ it follows $A_1\parallel A_2$. Then $A_2 = \alpha A_1$ for some $\alpha \in \C^\times$ so $A_1\perp A_2$ implies $A_1^2 = 0$. By Claim \ref{cl:preserves rank-one}, $\phi(A_1)$ has rank one, and it is also nilpotent since $\phi$ preserves spectrum. Therefore $\phi(A_1)^2 = 0$. By the homogeneity of $\phi$ (Claim \ref{cl:phi is homogeneous}), we conclude $\phi(A_1) \perp \alpha \phi(A_1) = \phi(\alpha A_1) = \phi(A_2)$, as desired.
		\end{claimproof}
		
		
		\begin{claim}\label{cl:parallel}
			We have $\phi(E_{ij}) \parallel E_{ij}$ or $\phi(E_{ij}) \parallel E_{ji}$ for all $(i,j) \in \rho^\times$.
		\end{claim}
		\begin{claimproof}
			By Lemma \ref{le:closure of R}, note that all matrix units of $\ca{A}_\rho$ are contained in $\overline{\ca{R}}$.
			For each $(i,j) \in \rho^\times$ we have $E_{ij}\perp E_{kk}$ for all $k \in [1,n]\setminus \{i,j\}$ so by Claims \ref{cl:identity on diagonals} and \ref{cl:preserves orthogonality} we obtain $$\supp \phi(E_{ij}) \subseteq \{i,j\} \times \{i,j\}.$$
			Via a direct computation we now show that for each $1 \le i \le n$ we have
			$$\abs{\rho(i)} \ge 3 \implies (\phi(E_{ij}) \parallel E_{ij}, \forall j \in \rho(i)) \,\text{ or }\, (\phi(E_{ij}) \parallel E_{ji}, \forall j \in \rho(i)).$$
			Indeed, if $(i,j),(i,k) \in \rho^\times$ for $j \ne k$, then there exist scalars $\alpha_{ii}, \alpha_{ij}, \alpha_{ji}, \alpha_{jj}, \beta_{ii}, \beta_{ik}, \beta_{ki}, \beta_{kk} \in \C$ such that
			\begin{align*}
				\phi(E_{ij}) &= \alpha_{ii}E_{ii} + \alpha_{ij}E_{ij} + \alpha_{ji}E_{ji} + \alpha_{jj}E_{jj},\\
				\phi(E_{ik}) &= \beta_{ii}E_{ii} + \beta_{ik}E_{ik} + \beta_{ki}E_{ki} + \beta_{kk}E_{kk}.
			\end{align*}
			Since $E_{ij} \perp E_{ik}$, by invoking Claim \ref{cl:preserves orthogonality}, we obtain 
			\begin{align*}
				0 &= \phi(E_{ij})\phi(E_{ik}) = \alpha_{ii}\beta_{ii}E_{ii} + \alpha_{ii}\beta_{ik}E_{ik} + \alpha_{ji}\beta_{ii}E_{ji} + \alpha_{ji}\beta_{ik}E_{jk},\\
				0 &= \phi(E_{ik})\phi(E_{ij}) = \alpha_{ii}\beta_{ii}E_{ii} + \alpha_{ij}\beta_{ii}E_{ij} + \alpha_{ij}\beta_{ki}E_{kj} + \alpha_{ii}\beta_{ki}E_{ki}.
			\end{align*}
			Suppose that $\beta_{ii} \ne 0$. Then $\alpha_{ii}\beta_{ii} = \alpha_{ij}\beta_{ii} = \alpha_{ji}\beta_{ii} = 0$ imply $\alpha_{ii} = \alpha_{ij} = \alpha_{ji} = 0$ and hence $\phi(E_{ij}) = \alpha_{jj}E_{jj}$, which contradicts injectivity (as $\phi(\alpha_{jj}E_{jj})=\alpha_{jj}E_{jj}$). We run into a similar contradiction when assuming $\alpha_{ii} \ne 0$ so we conclude $\alpha_{ii} = \beta_{ii} = 0$. Since $\phi(E_{ij})$ and $\phi(E_{ik})$ are rank-one nilpotents, we have
			$$0 = \phi(E_{ij})^2 = \alpha_{ij}\alpha_{ji}E_{ii} + \alpha_{ij}\alpha_{jj}E_{ij} + \alpha_{jj}\alpha_{ji}E_{ji} + (\alpha_{ij}\alpha_{ji} + \alpha_{jj}^2)E_{jj}$$
			and
			$$0 = \phi(E_{ik})^2 = \beta_{ik}\beta_{ki}E_{ii} + \beta_{ik}\beta_{kk}E_{ik} + \beta_{kk}\beta_{ki}E_{ki} + (\beta_{ik}\beta_{ki} + \beta_{kk}^2)E_{kk}.$$
			We first conclude $\alpha_{jj} = \beta_{kk} = 0$ and then $\alpha_{ij} = 0$ or $\alpha_{ji} = 0$ (but not both since otherwise we would have $\phi(E_{ij}) = 0$) and similarly $\beta_{ik} = 0$ or $\beta_{ki} = 0$ but not both. If $\alpha_{ij} \ne 0$, then $\phi(E_{ij}) = \alpha_{ij}E_{ij}$ and $\alpha_{ij}\beta_{ki} = 0$ implies $\beta_{ki} = 0$ from which we conclude $\phi(E_{ik}) = \beta_{ik}E_{ik}$. Similarly, if $\alpha_{ji} \ne 0$, we conclude $\phi(E_{ij}) = \alpha_{ji}E_{ji}$ and $\phi(E_{ik}) = \beta_{ki}E_{ki}$.
			
			\smallskip
			
			An analogous argument shows that for each $1 \le j \le n$ we have
			$$\abs{\rho^{-1}(j)} \ge 3 \implies (\phi(E_{ij}) \parallel E_{ij}, \forall i \in \rho^{-1}(j)) \,\text{ or }\, (\phi(E_{ij}) \parallel E_{ji}, \forall i \in \rho^{-1}(j)).$$
			It remains to consider the case $(i,j) \in \rho^{\times}$ when $\abs{\rho(i)} = 2$ (or $\abs{\rho^{-1}(j)} = 2$). For concreteness, assume $\rho(i) = \{i,j\}$ for some $j \in [1,n]\setminus\{i\}$. Since, by assumption,  $\ca{A}_\rho$ is 2-free, clearly there exists some $k \in [1,n]\setminus \{i,j\}$ such that $k \tripprox_0 i$ or $k \tripprox_0 j$. The possibilities $(i,k) \in \rho$ and $(j,k) \in \rho$ can be excluded, as they would lead to $k \in \rho(i)$, which is false. The possibilities which remain are $(k,i) \in \rho$ or $(k,j) \in \rho$, so in either case we can assume $(k,j) \in \rho$. Then $i,j,k \in \rho^{-1}(j)$ and hence $\abs{\rho^{-1}(j)} \ge 3$, which allows us to reach the desired conclusion that $\phi(E_{ij}) \parallel E_{ij}$ or $\phi(E_{ij}) \parallel E_{ji}$.\end{claimproof}
		
		In view of Claim \ref{cl:parallel}, for each $(i,j) \in \rho$, denote by $g(i,j) \in \C^\times$ the unique scalar such that $$\phi(E_{ij}) = g(i,j)E_{ij} \qquad\text{ or }\qquad\phi(E_{ij}) = g(i,j)E_{ji}.$$ In this manner we obtain a function $g : \rho \to \C^\times$ whose transitivity we intend to show in the remainder of the proof. As, by assumption, $\phi|_{\ca{D}_n}$ is the identity map, it is immediate that $g|_{\Delta_n} \equiv 1$. Define
		\begin{equation}\label{eq:definition of rho_M and rho_A}
			\rho^\phi_M := \{(i,j) \in \rho : \phi(E_{ij}) \parallel E_{ij}\}, \qquad \rho^\phi_A := \{(i,j) \in \rho : \phi(E_{ij}) \parallel E_{ji}\}.
		\end{equation}
		Clearly, $\rho^\phi_M \cup \rho^\phi_A = \rho$ and $\rho^\phi_M \cap \rho^\phi_A = \Delta_n.$
		\begin{claim}\label{cl:crucial property of rho_A and rho_M}
			Suppose $(i,j) \in \rho^\times$. If $(i,j) \in \rho^{\phi}_M$, then
			$$(\{i\} \times \rho(i)) \cup (\rho^{-1}(i) \times \{i\}) \cup (\{j\} \times \rho(j)) \cup (\rho^{-1}(j) \times \{j\}) \subseteq \rho^{\phi}_M.$$
			Similarly, if $(i,j) \in \rho^{\phi}_A$, then
			$$(\{i\} \times \rho(i)) \cup (\rho^{-1}(i) \times \{i\}) \cup (\{j\} \times \rho(j)) \cup (\rho^{-1}(j) \times \{j\}) \subseteq \rho^{\phi}_A.$$
		\end{claim}
		\begin{claimproof}
			For concreteness suppose that $(i,j) \in \rho^{\phi}_M$, as the other case is similar.
			\begin{itemize}
				\item If $k \in \rho^\times(i)$, then $E_{ij} \perp E_{ik}$ and Claim \ref{cl:preserves orthogonality} imply $(i,k) \in \rho^{\phi}_M$.
				\item If $k \in (\rho^\times)^{-1}(j)$, then $E_{ij} \perp E_{kj}$ and Claim \ref{cl:preserves orthogonality} imply $(k,j) \in \rho^{\phi}_M$.
				\item Let $k \in (\rho^\times)^{-1}(i)$. If $k \ne j$, then by transitivity we obtain $(k,j) \in \rho^\times$ and hence $(k,j) \in \rho^{\phi}_M$ and then $(k,i) \in \rho^{\phi}_M$, by the previous two cases. On the other hand, the case $k = j$ follows from the injectivity and homogeneity of $\phi$.
				\item Let $k \in \rho^\times(j)$. The case $k = i$ again follows from the injectivity and homogeneity of $\phi$. If $k \ne i$, then by transitivity we obtain $(i,k) \in \rho^\times$ and hence $(i,k) \in \rho^{\phi}_M$ and then $(j,k) \in \rho^{\phi}_M$, by the first two cases.
			\end{itemize}
		\end{claimproof}
		
		\begin{claim}\label{cl:phi preserves support} 
			Suppose that $X \in \ca{A}_\rho \cap M_n^{\subseteq S}$ for some $S \subseteq [1,n]$ (where $M_n^{\subseteq S}$ is defined in \eqref{eq:M_n subset}). Then $\phi(X) \in M_n^{\subseteq S}$.
		\end{claim}
		\begin{claimproof}
			Following our notation from Section \ref{sec:preliminaries}, by applying Lemma \ref{le:diagonals are dense} to the matrix $X^{\flat S^c} \in \ca{A}_\rho^{\flat S^c}$, by the continuity of $\phi$ it suffices to assume that $X$ is a diagonalizable matrix. Now the assertion follows from $X \perp E_{kk}$ for all $k \in [1,n]\setminus S$,  Claim \ref{cl:preserves zero-product} and $\phi(E_{kk}) = E_{kk}$.
		\end{claimproof}
		
		\begin{claim}\label{cl:restriction of phi}
			Let $S \subseteq [1,n]$ be a nonempty set. The map
			$$\psi : \ca{A}_\rho^{\flat S^c} \to M_{\abs{S}}, \qquad  X \mapsto \phi(X^{\sharp S^c})^{\flat S^c}$$
			is an injective continuous commutativity and spectrum preserver.
		\end{claim}
		\begin{claimproof}
			By Lemma \ref{le:deletion of SMA is SMA}, $\ca{A}_{\rho}^{\flat S^c}$ is an SMA in $M_{\abs{S}}$. In view of Claim \ref{cl:phi preserves support} it makes sense to consider the maps
			\begin{align*}
				&\psi_1 : \ca{A}_\rho^{\flat S^c} \to M_n^{\subseteq S}, \qquad X \mapsto X^{\sharp S^c},\\
				&\psi_2 : M_n^{\subseteq S} \to M_n^{\subseteq S}, \qquad X \mapsto \phi(X),\\
				&\psi_3 : M_n^{\subseteq S} \to M_{\abs{S}}, \qquad X \mapsto X^{\flat S^c}.
			\end{align*}
			Note that $\psi = \psi_3 \circ \psi_2 \circ \psi_1$. By Lemmas \ref{le:sharp is homomorphism} and \ref{le:flat is an isomorphism}, $\psi_1$ and $\psi_3$ are algebra monomorphisms so it follows that $\psi$ is an injective continuous commutativity preserver. Finally, for each $X \in \ca{A}_\rho^{\flat S^c}$  we have the equality of polynomials
			\begin{align*}
				x^{n-\abs{S}} k_{\psi(X)}(x) &= x^{n-\abs{S}} k_{\phi(X^{\sharp S^c})^{\flat S^c}}(x) =  k_{(\phi(X^{\sharp S^c})^{\flat S^c})^{\sharp S^c}}(x) \stackrel{\text{Claim }\ref{cl:phi preserves support}}= k_{\phi(X^{\sharp S^c})}(x) \\
				&\leftstackrel{\text{Claim }\ref{cl:preserves characteristic polynomial}}=   k_{X^{\sharp S^c}}(x)= x^{n-\abs{S}} k_{X}(x),
			\end{align*}
			which implies $k_{\psi(X)} = k_X$. In particular, $\psi$ is a spectrum preserver.
		\end{claimproof}
		
		It remains to show that $\rho^{\phi}_M$ and $\rho^{\phi}_A$ are quasi-orders and that the map $g$ is transitive. The reflexivity of $\rho^{\phi}_M$ and $\rho^{\phi}_A$ is immediate, while their transitivity follows from Claim \ref{cl:crucial property of rho_A and rho_M}. In order to show that $g$ is a transitive map, first note that Claim \ref{cl:crucial property of rho_A and rho_M} implies
		\begin{equation}\label{eq:transitive elements are in M or in A}
			(i,j),(j,k) \in \rho \qquad \implies \qquad ((i,j),(j,k) \in \rho^{\phi}_M \quad\text{ or }\quad (i,j),(j,k) \in \rho^{\phi}_A).
		\end{equation}
		Therefore, it suffices to show that $g|_{\rho^\phi_M}$ and $g|_{\rho^\phi_A}$ are transitive maps. We focus on $\rho^\phi_M$ as the proof for $\rho^\phi_A$ is analogous. Suppose that $(i,j),(j,k) \in \rho^\phi_M$. We show that $g(i,j)g(j,k) = g(i,k)$. This is obvious if $i = j$ or $j = k$, so assume further that  $(i,j),(j,k) \in (\rho^\phi_M)^\times$. We first focus on the case $i \ne k$. By deleting $l$-th row and column in $\ca{A}_\rho$ for each $l \in [1,n]\setminus \{i,j,k\}$, the elements which remain are
		$$\begin{bmatrix}
			(i,i) & (i,j) & (i,k) \\
			* & (j,j) & (j,k) \\
			* & * & (k,k)
		\end{bmatrix}.$$
		Therefore, $\ca{A}_\rho^{\flat(\{i,j,k\}^c)} \subseteq M_3$ contains $\ca{T}_3$. Hence, in view of Lemma \ref{le:deletion of SMA is SMA} and Claim \ref{cl:restriction of phi} we can apply Theorem \ref{thm:parabolic} to the map
		$$\psi : \ca{A}_\rho^{\flat(\{i,j,k\}^c)} \to M_3, \qquad X \mapsto \phi(X^{\sharp (\{i,j,k\}^c)})^{\flat(\{i,j,k\}^c)}$$
		to conclude that $\psi$ is a Jordan embedding (and thus multiplicative or antimultiplicative). As $(i,j),(j,k) \in \rho^\phi_M$, one easily verifies 
		$$\psi|_{\ca{D}_3} = \id, \qquad \psi(E_{12}) = g(i,j)E_{12}, \qquad \psi(E_{13}) = g(i,k)E_{13}, \qquad \psi(E_{23}) = g(j,k)E_{23},$$
		whence $\psi$ is a multiplicative map. In particular, we obtain
		\begin{align*}
			g(i,k)E_{13}& = \psi(E_{13}) = \psi(E_{12}E_{23}) = \psi(E_{12})\psi(E_{23})\\
			&= g(i,j)g(j,k)E_{13},
		\end{align*}
		yielding $g(i,k) = g(i,j)g(j,k)$.
		
		It remains to consider the shorter case $i = k$. Since $\ca{A}_\rho$ is 2-free, there exists some $l \in [1,n]\setminus \{i,j\}$ such that $i \tripprox_0 l$ or $j \tripprox_0 l$. Since $i$ and $j$ play a symmetric role, without loss of generality suppose the former. If $(i,l) \in \rho^\times$, then from $(j,i) \in \rho^{\phi}_M$ via Claim \ref{cl:crucial property of rho_A and rho_M} we conclude $(i,l) \in (\rho^{\phi}_M)^\times$. Now from $(j,i), (i,l) \in (\rho^{\phi}_M)^\times$ by the previous case it follows $(j,l) \in (\rho^{\phi}_M)^\times$ and
		$$g(j,l)= g(j,i)g(i,l), \qquad g(i,l) = g(i,j)g(j,l).$$
		We obtain
		$$g(i,j)g(j,i) = \frac{g(i,l)}{g(j,l)} \cdot \frac{g(j,l)}{g(i,l)} = 1$$
		as desired. On the other hand, if $(l,i) \in \rho^\times$ then similarly we obtain
		$$g(l,j) = g(l,i)g(i,j), \qquad g(l,i) = g(l,j)g(j,i),$$
		again yielding $g(i,j)g(j,i) = 1$. The proof of Proposition \ref{prop:necessarily multiplicative} is now complete.
	\end{proof}
	
	\begin{theorem}\label{thm:main result}
		Let $\ca{A}_\rho \subseteq M_n$ be an SMA. Then the following two statements are equivalent:
		\begin{enumerate}[(i)]
			\item For each $(i,j) \in \rho^\times$ we have 
			$$
			|(\rho(i) \cup \rho^{-1}(i)) \cap (\rho(j) \cup \rho^{-1}(j))| \geq 3.
			$$ 
			\item Every continuous injective commutativity and spectrum preserver $\phi : \ca{A}_\rho \to M_n$ is necessarily a Jordan embedding.
		\end{enumerate}
	\end{theorem}
	\begin{remark}\label{re:equivalent formulation of (i)}
		Note that for $(i,j) \in \rho^\times$, the condition (i) of Theorem \ref{thm:main result} (i) is equivalent to
		$$
		(\rho(i) \cup \rho^{-1}(i)) \cap (\rho(j) \cup \rho^{-1}(j)) \not \subseteq \{i,j\}.
		$$
		In particular, (i) implies that $\ca{A}_\rho$ is $2$-free. 
	\end{remark}

	In the next lemma we first consider the $n=3$ case.
	\begin{lemma}\label{le:(i) iff parabolic}
		Let $\ca{A}_\rho \subseteq M_3$ be an SMA distinct from $\ca{D}_3$. The following statements are equivalent:
		\begin{enumerate}[(a)]
			\item $\ca{A}_\rho$ satisfies the condition (i) from Theorem \ref{thm:main result}.
			\item There exist distinct $i,j \in [1,3]$ such that $\rho(i) = \rho^{-1}(j) = [1,3]$.
			\item There exists a permutation $\pi \in S_3$ such that $R_{\pi}\ca{A}_\rho R_{\pi}^{-1}$ is a block upper-triangular subalgebra of $M_3$.
		\end{enumerate}
	\end{lemma}
	\begin{proof}
		For $\pi \in S_3$, denote the quasi-order $$\rho' := \{(\pi(i),\pi(j)) : (i,j) \in \rho\}.$$
		Note that for all $i,j \in [1,3]$ we have 
		\begin{equation}\label{eq:rho equivalence}
			\rho(i) = \rho^{-1}(j) = [1,3] \iff \rho'(\pi(i)) = (\rho')^{-1}(\pi(j)) = [1,3].
		\end{equation}
		Notice also that an SMA in $M_3$ is a block upper-triangular subalgebra if and only if it contains the first row and third column.
		
		\smallskip
		
		\noindent \fbox{$(a)\implies (b)$} Since $\ca{A}_\rho \ne \ca{D}_3$, choose some $(i,j) \in \rho^\times$. By (i), there exists $k \in [1,3] \setminus\{i,j\}$ such that
		$$k \in (\rho(i) \cup \rho^{-1}(i)) \cap (\rho(j) \cup \rho^{-1}(j)).$$
		We consider two cases:
		\begin{itemize}
			\item If $k \in \rho(j)$, then from $(i,j),(j,k) \in \rho$ it follows $(i,k) \in \rho$ and hence $\rho(i) =\rho^{-1}(k) = [1,3]$.
			\item If $k \in \rho^{-1}(j)$, then from $(i,j),(k,j) \in \rho$ it follows $\rho^{-1}(j) = [1,3]$. Moreover, since $k \in \rho(i) \cup \rho^{-1}(i)$, we obtain $(i,k) \in \rho$ or $(k,i) \in \rho$, which implies $\rho(i) = [1,3]$ or $\rho(k) = [1,3]$, respectively.
		\end{itemize}
		
		\smallskip
		
		\noindent \fbox{$(b) \implies (c)$} Let $\pi \in S_3$ be any permutation such that $\pi(i) = 1$ and $\pi(j) = 3$. Then by \eqref{eq:rho equivalence}, $R_{\pi}\ca{A}_\rho R_{\pi}^{-1} = \ca{A}_{\rho'}$ satisfies $\rho'(1) = (\rho')^{-1}(3) = [1,3]$, so $R_{\pi}\ca{A}_\rho R_{\pi}^{-1}$ is a block upper-triangular subalgebra of $M_3$.
		
		\smallskip
		
		\noindent \fbox{$(c) \implies (a)$} By assumption, there exists a permutation $\pi \in S_3$ such that $\rho'(1) = (\rho')^{-1}(3) = [1,3]$. Denote $i := \pi^{-1}(1)$ and $j := \pi^{-1}(3)$. Again, by \eqref{eq:rho equivalence}, we have $\rho(i) = \rho^{-1}(j) = [1,3].$
		Let $k \in [1,3]\setminus \{i,j\}$. Since $k \in \rho(i)$ and $k \in \rho^{-1}(j)$, we also have $\rho(k) \cup \rho^{-1}(k) = [1,3]$. Therefore, $\ca{A}_\rho$ satisfies (i).
	\end{proof}
	
	\color{black}
	\begin{proof}[Proof of Theorem {\ref{thm:main result}}]\phantom{x}
		
		\smallskip
		
		\noindent \fbox{$(i)\implies(ii)$} First note that the $n=1$ case is trivial, while the only SMA in $M_2$ which satisfies $(i)$ is $\ca{D}_2$. In the $n=3$ case, this implication follows directly from Lemma \ref{le:(i) iff parabolic}, Theorem \ref{thm:parabolic} and Remark \ref{re:invariant under similarity}. In the remainder of the proof, we can therefore assume $n \ge 4$ (although this is not necessary).
		
		\smallskip
		
		\noindent Since $\ca{A}_\rho$ is 2-free, in view of Proposition \ref{prop:necessarily multiplicative} (and Claim \ref{cl:phi is homogeneous}) without loss of generality by passing to the map $S^{-1}\phi((g^*)^{-1}(\cdot))S$ (where $g^*$ is the automorphism induced by $g$, as defined in \eqref{eq:inducedauto}) we can assume that
		\begin{equation}\label{eq:after conjugation}
			\phi(E_{ij}) = \begin{cases}
				E_{ij}, &\text{ if } (i,j) \in \rho_M^{\phi},\\
				E_{ji}, &\text{ if } (i,j) \in \rho_A^{\phi}
			\end{cases}
		\end{equation}
		where $\rho_M^{\phi}$ and $\rho_A^{\phi}$ are quasi-orders on $[1,n]$ defined by \eqref{eq:definition of rho_M and rho_A}.
		
		\begin{claim}\label{cl:rho_M and rho_A satisfy (i)}
			Both quasi-orders $\rho_M^{\phi}$ and $\rho_A^{\phi}$ satisfy (i). In particular, by Remark \ref{re:equivalent formulation of (i)}, both SMAs $\ca{A}_{\rho_M^{\phi}}$ and $\ca{A}_{\rho_A^{\phi}}$ are $2$-free.
		\end{claim}
		\begin{claimproof}
			For simplicity, we prove the claim for $\rho_M^{\phi}$, as the argument for $\rho_A^{\phi}$ is almost identical. Suppose that $(i,j) \in (\rho_M^{\phi})^\times$. In particular, $(i,j) \in \rho^\times$ so by Remark \ref{re:equivalent formulation of (i)} there exists some $k \in [1,n]\setminus\{i,j\}$ such that
			$$k \in (\rho(i) \cup \rho^{-1}(i)) \cap (\rho(j) \cup \rho^{-1}(j)).$$
			Then $(i,k) \in \rho$ or $(k,i) \in \rho$ and similarly $(j,k)\in \rho$ or  $(k,j) \in \rho$. In either case, Claim \ref{cl:crucial property of rho_A and rho_M} implies that the corresponding elements of $\rho$ are in fact contained in $\rho_M^{\phi}$. Thus,
			$$k \in (\rho_M^{\phi}(i) \cup (\rho_M^{\phi})^{-1}(i)) \cap (\rho_M^{\phi}(j) \cup (\rho_M^{\phi})^{-1}(j)).$$
			
		\end{claimproof}
		
		\begin{claim}\label{cl:every rank-one matrix is in rhoM or rhoA}
			Every rank-one matrix in $\ca{A}_\rho$ is contained in $\ca{A}_{\rho_M^{\phi}}$ or in $\ca{A}_{\rho_A^{\phi}}$.
		\end{claim}
		\begin{claimproof}
			Assume $ab^* \in \ca{A}_\rho$ for some nonzero vectors $a,b \in \C^n$. We have
			$$(\supp a) \times (\supp b) \subseteq \rho.$$
			Let $(i,j),(k,l) \in ((\supp a) \times (\supp b)) \mathop{\cap} \rho^\times$ be distinct but otherwise arbitrary. We need to show that $(i,j),(k,l) \in \rho_M^{\phi}$ or $(i,j),(k,l) \in \rho_A^{\phi}$. If $j = k$ or $i = l$, this follows from \eqref{eq:transitive elements are in M or in A}. Otherwise, since $\{i,k\} \times \{j,l\} \subseteq \rho$, we have
			$$\begin{cases}
				(i,l) \in \rho^\times \implies (i,j),(i,l),(k,l) \in \rho^\times,\\
				(k,j) \in \rho^\times \implies (i,j),(k,j),(k,l) \in \rho^\times,
			\end{cases}$$
			so in either case we obtain the desired assertion from Claim \ref{cl:crucial property of rho_A and rho_M}.
		\end{claimproof}

		\begin{claim}\label{cl:identity on R}
			$\phi$ acts as the identity on $\ca{R} \cap \ca{A}_{\rho_M^{\phi}}$, and as transposition on $\ca{R} \cap \ca{A}_{\rho_A^{\phi}}$.
		\end{claim}
		\begin{claimproof}
			For concreteness assume that $ab^* \in \ca{R} \cap \ca{A}_{\rho_M^{\phi}}$ for some (nonzero) vectors $a,b \in \C^n$, as the other case is similar. In view of Claim \ref{cl:preserves rank-one}, denote $\phi(ab^*) = xy^*$ for some nonzero $x,y \in \C^n$. As $b^*a \ne 0$, we can fix some $$j \in (\supp a) \cap (\supp b).$$ Let $i \in [1,n]\setminus \{j\}$ be arbitrary and consider the matrix
			$$A := (\overline{b_j}e_i - \overline{b_i}e_j)(\overline{a_i}e_j - \overline{a_j}e_i)^*. $$
			Note that
			$$(\overline{b_j}e_i - \overline{b_i}e_j)e_i^* = \begin{cases}
				\overline{b_j}E_{ii} - \overline{b_i} E_{ji}, &\quad\text{ if }i \in \supp b,\\
				\overline{b_j}E_{ii}, &\quad\text{ if }i \notin \supp b,
			\end{cases}$$
			$$e_i (\overline{a_i}e_j - \overline{a_j}e_i)^* = \begin{cases}
				a_iE_{ij}-a_jE_{ii}, &\quad\text{ if }i \in \supp a,\\
				-a_jE_{ii}, &\quad\text{ if }i \notin \supp a.
			\end{cases}$$
			In all cases, these matrices belong to $\ca{A}_{\rho}$, so by Lemma \ref{le:closure of R} we have $A \in \overline{\ca{R}}$ and in particular $A \in \ca{A}_\rho$.
			
			\smallskip
			
			We claim that $\phi(A) = A$. Consider the following cases:
			\begin{itemize}
				\item If $i \notin (\supp a) \cup (\supp b)$, then $A = -a_j\overline{b_j} E_{ii}$ and hence $\phi(A) = A$ by the homogeneity of $\phi$ (Claim \ref{cl:phi is homogeneous}).
				\item Suppose that $i \in \supp a$. Then $(i,j) \in \supp A \subseteq \{i,j\}\times\{i,j\}$, so $(i,j) \in \rho^\times$. By (i) there exists $k \in [1,n]\setminus \{i,j\}$ such that
				$$k \in (\rho(i) \cup \rho^{-1}(i)) \cap (\rho(j) \cup \rho^{-1}(j)).$$
				It is immediate that
				$$\{i,j,k\} \subseteq (\rho(i) \cup \rho^{-1}(i)) \cap (\rho(j) \cup \rho^{-1}(j)) \cap (\rho(k) \cup \rho^{-1}(k)),$$
				which via Lemma \ref{le:deletion of SMA is SMA} implies that the SMA $\ca{A}_\rho^{\flat(\{i,j,k\}^c)} \subseteq M_3$ satisfies (i). Furthermore, since $$(i,j) \in (\supp a) \times (\supp b) = \supp(ab^*) \subseteq \rho_M^{\phi},$$
				by Claim \ref{cl:crucial property of rho_A and rho_M} we conclude that
				$$(\{i,j,k\} \times \{i,j,k\}) \cap \rho \subseteq \rho_M^{\phi}.$$
				In particular, by \eqref{eq:after conjugation}, $\phi$ acts as the identity on all matrix units supported in $(\{i,j,k\} \times \{i,j,k\}) \cap \rho$. We can now invoke Lemma \ref{le:deletion of SMA is SMA} and Claim \ref{cl:restriction of phi} with the $n=3$ case (which was already covered) to conclude that the map
				$$\psi : \ca{A}_\rho^{\flat(\{i,j,k\}^c)} \to M_3, \qquad X \mapsto \phi(X^{\sharp (\{i,j,k\}^c)})^{\flat(\{i,j,k\}^c)}$$
				is a Jordan embedding. Since $\psi$ acts as the identity on all matrix units of $\ca{A}_\rho^{\flat(\{i,j,k\}^c)}$, we conclude that $\psi$ is the identity map. In particular, Claim \ref{cl:phi preserves support} implies that $\phi(A) = A$.
				\item Suppose that $i \in \supp b$. Then $(j,i) \in \rho^\times$ so by the symmetry of our assumption (i), the exact same discussion as above yields $\phi(A) = A$.
			\end{itemize}
			Now notice that $ab^* \perp A$ so by Claim \ref{cl:preserves orthogonality} we obtain
			$$xy^* = \phi(ab^*) \perp \phi(A) = A \implies (\overline{a_i}e_j - \overline{a_j}e_i)^*x = y^*(\overline{b_j}e_i - \overline{b_i}e_j) = 0.$$
			Overall, it follows
			$$x \perp \{\overline{a_i}e_j - \overline{a_j}e_i : i \in [1,n]\setminus \{j\}\}, \qquad y \perp \{\overline{b_j}e_i - \overline{b_i}e_j : i \in [1,n]\setminus \{j\}\}.$$
			In fact, these sets are bases for $\{a\}^\perp$ and $\{b\}^\perp$ respectively. We conclude $x\parallel a$ and $y \parallel b$, which implies $\phi(ab^*) \parallel ab^*$. Equating the traces yields $\phi(ab^*) = ab^*$.
		\end{claimproof}
		
		\begin{claim}\label{cl:penultimate step}
			$\phi$ acts as the identity on $\ca{A}_{\rho_M^{\phi}}$, and as transposition on $\ca{A}_{\rho_A^{\phi}}$.
		\end{claim}
		\begin{claimproof}
			For the sake of variety, we prove the second claim, as the first one is similar. By Claim \ref{cl:preserves diagonalizability}, for each $S \in \ca{A}_{\rho_A^{\phi}}^\times$ there exists $T \in M_n^\times$ such that 
			$$\phi(SDS^{-1}) = TDT^{-1}, \qquad \text{ for all }D \in \ca{D}_n.$$
			In particular, for all $j \in [1,n]$ we have
			$$TE_{jj}T^{-1} = \phi(\underbrace{SE_{jj}S^{-1}}_{\in \ca{R} \cap \ca{A}_{\rho_A^{\phi}}}) \stackrel{\text{Claim } \ref{cl:identity on R}}= (SE_{jj}S^{-1})^t.$$
			Hence, by the linearity of the maps $T(\cdot)T^{-1}$ and $(S(\cdot)S^{-1})^t$, for all $D \in \ca{D}_n$ we have
			$$\phi(SDS^{-1}) = TDT^{-1} = (SDS^{-1})^t.$$
			The Claim now follows by the continuity of $\phi$ from Lemma \ref{le:diagonals are dense} applied to $\ca{A}_{\rho_A^{\phi}}$.
		\end{claimproof}
		
		\begin{claim}\label{cl:the idempotent corresponding to rho_M}
			Let $P \in \ca{D}_n$ be a diagonal idempotent defined by
			$$P_{ii} = 1 \iff \text{there exists $j \in [1,n] \setminus \{i\}$ such that $(i,j) \in \rho^\phi_M$ or $(j,i) \in \rho^\phi_M$}.$$
			Then $P\in Z(\ca{A}_\rho)$, $PX \in \ca{A}_{\rho^\phi_M}$ and $(I-P)X \in \ca{A}_{\rho^\phi_A}$ for all $X \in \ca{A}_\rho$.
		\end{claim}
		\begin{claimproof}
			A variant of this argument (when $\phi$ is assumed to be a Jordan homomorphism) already appears in \cite[Lemma 4.8 (c)]{GogicTomasevic} and a similar proof applies here. Indeed, it suffices to show that for all $(i,j) \in \rho$ we have
			\begin{equation}\label{eq:three claims}
				P \leftrightarrow E_{ij}, \qquad PE_{ij} \in \ca{A}_{\rho^\phi_M}, \qquad (I-P)E_{ij} \in \ca{A}_{\rho^\phi_A}.
			\end{equation} Since all three claims are trivially true when $i=j$, fix $(i,j) \in \rho^\times$. We consider two separate cases:
			\begin{itemize}
				\item If $(i,j) \in \rho^\phi_M$, then $P_{ii} = P_{jj} = 1$ by definition, so
				$$PE_{ij} = \underbrace{E_{ij}}_{\in \ca{A}_{\rho^\phi_M}} = E_{ij}P \implies (I-P)E_{ij} = 0 \in \ca{A}_{\rho^\phi_A},$$
				which establishes \eqref{eq:three claims}.
				
				\item If $(i,j) \in \rho^\phi_A$, then by Claim \ref{cl:crucial property of rho_A and rho_M}, $P_{ii} = P_{jj} = 0$. Hence,
				$$PE_{ij} = \underbrace{0}_{\in \ca{A}_{\rho^\phi_M}} = E_{ij}P \implies (I-P)E_{ij} = E_{ij} \in \ca{A}_{\rho^\phi_A},$$
				so again \eqref{eq:three claims} follows.
			\end{itemize}
		\end{claimproof}
		\begin{claim}
			Let $P\in Z(\ca{A}_\rho)$ be the idempotent from Claim \ref{cl:the idempotent corresponding to rho_M}. Then for all $X \in \ca{A}_\rho$ we have
			$$\phi(X) = PX + (I-P)X^t.$$
			In particular, $\phi$ is a Jordan embedding.
		\end{claim}
		\begin{claimproof}
			Fix a diagonalizable matrix $X \in \ca{A}_\rho$. Since the idempotent $P$ is central, from Theorem \ref{thm:inner diagonalization on SMA} it easily follows that $PX$ and $(I-P)X$ are both (in fact, simultaneously) diagonalizable. By Claim \ref{cl:the idempotent corresponding to rho_M}, we have $PX \in \ca{A}_{\rho^\phi_M}$, $(I-P)X \in \ca{A}_{\rho^\phi_A}$ and trivially $PX \perp (I-P)X$. Therefore, we have
			\begin{align*}
				\phi(X) &= \phi(PX + (I-P)X) \stackrel{\text{ Claim \ref{cl:phi is homogeneous}}}= \phi(PX) + \phi((I-P)X) \stackrel{\text{ Claim \ref{cl:penultimate step}}}= PX + ((I-P)X)^t \\
				&= PX + (I-P)X^t.
			\end{align*}
			It follows that the continuous maps $\phi$ and $P(\cdot) + (I-P)(\cdot)^t$ coincide on the set of all diagonalizable matrices in $\ca{A}_\rho$, so Lemma \ref{le:diagonals are dense} implies their overall equality. 
		\end{claimproof}
		
		\noindent The proof of $(i)\implies(ii)$ is now complete.
		
		\smallskip
		
		\noindent \fbox{$(ii)\implies (i)$} 
		Suppose that an SMA $\ca{A}_{\rho} \subseteq M_n$ fails to satisfy (i). Then by Remark \ref{re:equivalent formulation of (i)} there exists $(r,s) \in \rho^\times$ such that
		$$(\rho(r) \cup \rho^{-1}(r)) \cap (\rho(s) \cup \rho^{-1}(s)) = \{r,s\}.$$
		We consider two separate cases:
		
		\smallskip
		\setcounter{case}{0}
		\begin{case}
			Suppose that $(s,r) \in \rho$. Then by the transitivity of $\rho$ it easily follows 
			\begin{equation}\label{eq:rth row and column}
				\rho(r) =\rho^{-1}(r) = \rho(s) = \rho^{-1}(s) = \{r,s\}.
			\end{equation}
			Define a rank-two idempotent 
			$$P := E_{rr} + E_{ss} \in \ca{A}_{\rho}.$$
			Note that $P$ is central. Indeed, fix $(i,j) \in \rho$ and note that $i \in \{r,s\}$ if and only if $j \in \{ r,s\}$. On the other hand, we clearly have
			$$PE_{ij} = \begin{cases}
				E_{ij}, &\text{ if } i \in \{r,s\}, \\
				0, &\text{ if } i \notin \{r,s\},
			\end{cases} \qquad E_{ij}P = \begin{cases}
				E_{ij}, &\text{ if } j \in \{r,s\}, \\
				0, &\text{ if } j \notin \{r,s\},
			\end{cases}$$
			so $P \leftrightarrow E_{ij}$. If follows that $$\ca{A}_{\rho} = P \ca{A}_{\rho} \oplus (I - P)\ca{A}_\rho$$
			is an inner direct sum of algebras. Moreover, by \eqref{eq:rth row and column}, $P \ca{A}_{\rho} $ is isomorphic to $M_2$ so let $\psi : M_2 \to M_2$ be a slight modification of the nonlinear counterexample due to \cite[Example 7]{PetekSemrl} (in order to ensure its injectivity). Let $f : [0,+\infty) \to \Sph^1$ be any nonconstant continuous map such that $\lim_{t \to + \infty} f(t)=1$ (concretely, we can choose $f(t):= e^{\frac{i\pi}{t+1}}$). Define $\psi : M_2 \to M_2$ by
			$$\psi\left(\begin{bmatrix} a & b \\ c & d\end{bmatrix}\right):= \left\{\begin{array}{cc}
				\begin{bmatrix} a & 0 \\ c & d\end{bmatrix}, & \text{ if } b= 0, \\
				\begin{bmatrix} a & b\,f\left(\abs{\frac{c}b}\right) \\ c\,\overline{f\left(\abs{\frac{c}b}\right)} & d\end{bmatrix}, & \text{ otherwise.}
			\end{array}\right.$$
			Then $\psi$ is a nonlinear injective continuous spectrum and commutativity preserver (and hence not a Jordan homomorphism). Then the same holds for the map
			$$\phi : \ca{A}_{\rho} \to M_n, \qquad \phi(X) := \underbrace{\psi(X^{\flat\{r,s\}^c})^{\sharp\{r,s\}^c}}_{\in P\ca{A}_\rho} + \underbrace{(X^{\flat\{r,s\}})^{\sharp\{r,s\}}}_{=(I-P)X}.$$
		\end{case}
		
		\begin{case} Suppose that $(s,r) \notin \rho$. By the assumption and the transitivity of $\rho$ it easily follows that \begin{equation}\label{eq:r and s}
				\rho^{-1}(r) = \{r\}, \qquad \rho(s) = \{s\}.
			\end{equation} Each $X \in \ca{A}_\rho$ can be written as $$X = X^\standardstate + X_{rs}E_{rs},$$ where $$X^\standardstate := X - X_{rs}E_{rs} \in \ca{A}_{\rho}.$$        
			Consider the map $$f : \C \times \C \to \C, \qquad f(u,v):= \begin{cases}
				v, \qquad &\text{ if }\abs{u} \le \abs{v},\\
				v\abs{\frac{v}{u}}, \qquad &\text{ if }\abs{u} > \abs{v}.
			\end{cases}$$
			It is straightforward to check that $f$ is continuous, homogeneous, and that $f(u,\cdot) : \C \to \C$ is injective for each fixed $u \in \C$. Define a map 
			$$\phi : \ca{A}_\rho \to \ca{A}_\rho, \qquad \phi(X) := X^\standardstate + f(X_{ss}-X_{rr},X_{rs})E_{rs},$$
			that is
			$$\phi(X)_{ij} = \begin{cases}
				X_{ij}, &\quad\text{ if }(i,j) \ne (r,s),\\
				f(X_{ss}-X_{rr},X_{rs}), &\quad\text{ if }(i,j) = (r,s).\\
			\end{cases}$$
			We claim that $\phi$ is a continuous injective commutativity and spectrum preserver, but not a linear map (and hence not a Jordan homomorphism).
			\begin{itemize}
				\item The continuity of $\phi$ follows directly from the continuity of $f$.
				\item Using the Laplace expansion along the $s$-th row and the $r$-th column, we obtain
				$$k_X(x) = (X_{rr}-x)(X_{ss}-x)k_{X^{\flat{\{r,s\}}}}(x) = k_{X^\standardstate}(x).$$
				In particular, the spectrum of $X$ is equal to the spectrum of $X^\standardstate$. As $\phi(X)^\standardstate = X^\standardstate$, we conclude that $\phi$ is a spectrum preserver.
				\item Suppose that $\phi(X) = \phi(Y)$ for some $X,Y \in \ca{A}_\rho$. Immediately we obtain 
				\begin{enumerate}[(i)]
					\item $X^\standardstate = Y^\standardstate$,
					\item $f(X_{ss}-X_{rr},X_{rs}) = f(Y_{ss}-Y_{rr},Y_{rs})$,
				\end{enumerate}
				whence we conclude $f(X_{ss}-X_{rr},X_{rs}) = f(X_{ss}-X_{rr},Y_{rs})$.  This implies $X_{rs} = Y_{rs}$, since $f(X_{ss}-X_{rr},\cdot)$ is injective. Therefore, $X = Y$, so $\phi$ is injective.
				\item Suppose that $X,Y \in \ca{A}_\rho$. We have
				\begin{align*}
					XY &= (X^\standardstate + X_{rs}E_{rs})(Y^\standardstate + Y_{rs}E_{rs}) \\
					&= X^\standardstate Y^\standardstate + Y_{rs}X^\standardstate E_{rs} + X_{rs}E_{rs}Y^\standardstate\\
					&= X^\standardstate Y^\standardstate + Y_{rs}\left(\sum_{i\in\rho^{-1}(r)}X_{ir}E_{ir}\right) E_{rs} + X_{rs}E_{rs}\left(\sum_{i\in\rho(s)}Y_{si}E_{si}\right)\\
					&\leftstackrel{\eqref{eq:r and s}}= X^\standardstate Y^\standardstate + (X_{rr}Y_{rs} + X_{rs}Y_{ss})E_{rs}.
				\end{align*}
				Moreover, $(XY)^\standardstate = X^\standardstate Y^\standardstate$, as
				$E_{rr} (X^\standardstate Y^\standardstate) E_{ss} = 0$. Indeed, if $E_{rr} E_{ij}  E_{kl}  E_{ss}\ne 0$ for some $(i,j),(k,l) \in \rho \setminus \{(r,s)\}$, then $i = r$, $j = k$ and $l = s$, which implies that $j \in \rho(r) \cap \rho^{-1}(s) = \{r,s\}$; a contradiction. Similarly,
				$$YX = Y^\standardstate X^\standardstate + (Y_{rr}X_{rs} + Y_{rs}X_{ss})E_{rs}, \qquad (YX)^\standardstate = Y^\standardstate X^\standardstate.$$
				Hence,
				\begin{equation}\label{eq:commutativity criterion}
					X \leftrightarrow Y \iff \begin{cases}
						X^\standardstate \leftrightarrow Y^\standardstate,\\
						(X_{ss}-X_{rr})Y_{rs} = (Y_{ss}-Y_{rr})X_{rs}.
					\end{cases}
				\end{equation}
				Assume now $X \leftrightarrow Y$. Since $\phi(X)^\standardstate = X^{\standardstate}$ and $\phi(Y)^\standardstate = Y^{\standardstate}$, to show that $\phi(X) \leftrightarrow \phi(Y)$, it remains to verify that
				$$(X_{ss}-X_{rr})f(Y_{ss}-Y_{rr},Y_{rs}) = (Y_{ss}-Y_{rr})f(X_{ss}-X_{rr},X_{rs}).$$
				By the homogeneity of $f$, this is equivalent to
				$$f((X_{ss}-X_{rr})(Y_{ss}-Y_{rr}),(X_{ss}-X_{rr})Y_{rs}) = f((X_{ss}-X_{rr})(Y_{ss}-Y_{rr}),(Y_{ss}-Y_{rr})X_{rs}),$$
				which is true by \eqref{eq:commutativity criterion}. We conclude that $\phi$ preserves commutativity.
				\item That $\phi$ is not an additive map follows from
				$$\phi(2E_{rr}+E_{rs}) = 2E_{rr}+\frac12 E_{rs}, \qquad \phi(E_{rs}) = E_{rs}, \qquad \phi(2E_{rr}+2E_{rs}) = 2E_{rr}+2 E_{rs}.$$
			\end{itemize}
		\end{case}
		\noindent The proof of Theorem \ref{thm:main result} is now complete.
	\end{proof}
	
	\section{Examples and final remarks}\label{sec:examplesandremarks}
	In contrast to the previous cases of $M_n$, $\ca{T}_n$, and the general block upper-triangular subalgebras of $M_n$, note that for an arbitrary SMA $\ca{A}_\rho \subseteq M_n$, Jordan embeddings $\ca{A}_\rho \to M_n$ are not necessarily multiplicative or antimultiplicative, even when $\ca{A}_\rho$ satisfies the condition (i) of Theorem \ref{thm:main result}.
	
	\begin{example}\label{ex:noMAMP}
		Consider the quasi-order on $[1,6]$ defined by
		$$\rho := ([1,3]\times [1,3]) \cup ([4,6]\times [4,6]).$$
		Clearly, $\ca{A}_\rho = \diag(M_3,M_3)\subseteq M_6$ satisfies the condition (i) of Theorem \ref{thm:main result}. However, the Jordan automorphism $\phi$ of $\ca{A}_\rho$ given by 
		$
		\phi(\diag(X,Y)) := \diag(X,Y^t)
		$ is obviously neither multiplicative nor antimultiplicative.
	\end{example}
	Further, when $\mathcal{A}_\rho$ is a block-upper triangular subalgebra of $M_n$, it is easy to see that $\rho$ admits only trivial transitive maps. In that case by Theorem \ref{thm:Jordan embeddings of SMAs}  (or Theorem \ref{thm:parabolic}), all Jordan embeddings $\mathcal{A}_\rho \to M_n$ are necessarily rank(-one) preservers. This is no longer true for general SMAs $\ca{A}_\rho \subseteq M_n$ which satisfy the condition (i) of Theorem \ref{thm:main result}. 
	
	\begin{example}\label{ex:SMA-nonlinearOK-nontrivialtransitive}
		Consider the quasi-order
		$$\rho := \Delta_7 \cup ([1,3]\times [4,7]) \cup \{(1,3),(4,5),(6,7)\}$$
		on $[1,7]$ and the corresponding SMA
		$$\ca{A}_\rho := \begin{bmatrix}
			* &  0 &  * &  * &  * &  *  & *  \\
			0  & *  & 0 &  *  & *  & *  & *  \\
			0  & 0 &  * &  * &  * &  *  & *  \\
			0  & 0  & 0  & *  & *  & 0  & 0  \\
			0  & 0  & 0 &  0 &  * &  0 &  0  \\
			0 &  0  & 0  & 0  & 0  & *  & *  \\
			0 &  0  & 0  & 0  & 0  & 0  & *  
		\end{bmatrix} \subseteq \ca{T}_7.$$
		We have
		\begin{align*}
			\rho(1) \cup \rho^{-1}(1) &= \rho(3) \cup \rho^{-1}(3) = \{1,3,4,5,6,7\},\\
			\rho(2) \cup \rho^{-1}(2) &= \{2,4,5,6,7\},\\
			\rho(4) \cup \rho^{-1}(4) &= \rho(5) \cup \rho^{-1}(5) = \{1,2,3,4,5\},\\
			\rho(6) \cup \rho^{-1}(6) &= \rho(7) \cup \rho^{-1}(7) = \{1,2,3,6,7\}.
		\end{align*}
		One easily checks that $\ca{A}_\rho$ satisfies (i) of Theorem \ref{thm:main result} and that the map $$g : \rho \to \C^\times, \qquad g(i,j) := \begin{cases}
			2, &\quad \text{ if } (i,j) \in \{(2,4),(2,5)\},\\
			1, &\quad \text{ otherwise}
		\end{cases}$$
		is transitive. On the other hand, $g$ is nontrivial. Indeed, if $g$ separates through the map $s: [1,7]\to \C^\times$, we obtain
		$$
		\frac{s(1)}{s(4)}=g(1,4)=1, \qquad \frac{s(1)}{s(6)}=g(1,6)=1 \qquad \implies \qquad s(4)=s(1)=s(6),
		$$
		so
		$$
		2=g(2,4)=\frac{s(2)}{s(4)}=\frac{s(2)}{s(6)}=g(2,6)=1,
		$$
		which is a contradiction.
		
		\smallskip
		
		\noindent Further, as for the induced algebra automorphism $g^*$ of $\ca{A}_\rho$ we have
		$$
		g^*(E_{14}+E_{16}+E_{24}+E_{26})=E_{14}+E_{16}+2E_{24}+E_{26},
		$$
		it is clear that $g^*$ does not preserve rank-one matrices (see also \cite[Theorem~5.7]{GogicTomasevic}). 
	\end{example}
	
	\begin{remark}\label{rem:indispensability}
		For an SMA $\ca{A}_\rho \subseteq M_n$ which satisfies (i) of Theorem \ref{thm:main result}, we discuss the necessity of all assumptions in (ii). In view of the first paragraph of the proof of (i) $\implies$ (ii) in Theorem \ref{thm:main result}, we assume $n \geq 3$.
		\begin{itemize}
			\item Spectrum preserving is necessary to assume for all SMAs $\ca{A}_\rho$, as the map $\phi : \ca{A}_\rho \to \ca{A}_\rho$, $\phi(X):=2X$ shows (for a nonlinear variant of the example see \cite[Example 5.1]{GogicPetekTomasevic}).
			\item Commutativity preserving is necessary to assume for all SMAs $\ca{A}_\rho$. Indeed, when $\ca{A}_\rho \ne \ca{D}_n$, this follows from a small modification of \cite[Example 5.2]{GogicPetekTomasevic}, by defining the map $\phi : \mathcal{A}_\rho \to M_n$,
			$$\phi(X):=\diag(1,\ldots,1,e^{\det X}, 1, \ldots, 1) X \diag(1,\ldots,1,e^{-\det X}, 1, \ldots, 1),$$ where both $e^{\det X}$ and $e^{-\det X}$ stand at some position $i \in [1,n]$ for which $\rho^\times(i)\neq \emptyset$. On the other hand, when $\ca{A}_\rho = \ca{D}_n$, we can consider the map
			$$\phi : \ca{D}_n \to M_n, \qquad \phi(X) := X + X_{11}E_{2n},$$
			which is clearly a linear (hence continuous) injective spectrum preserver, but not a Jordan homomorphism.
			\item Continuity is necessary to assume for all SMAs $\ca{A}_\rho$. This follows from \cite[Example 5.3]{GogicPetekTomasevic}, without any change.
			\item Injectivity is necessary to assume if and only if $\ca{A}_\rho$ is not semisimple (note that by \cite{Coelho}, $\ca{A}_\rho$ is semisimple if and only if $\rho$ is symmetric). Indeed, suppose that $\ca{A}_\rho$ is semisimple and that $\phi : \ca{A}_\rho \to M_n$ is a continuous commutativity and spectrum preserver. In view of Remark \ref{re:invariant under similarity} and 
			\cite[p.\ 432]{Akkurt} (see also \cite[Lemma 3.2]{GogicTomasevic}), we can assume that $\ca{A}_\rho = \diag(M_{k_1}, \ldots, M_{k_p})$ where, by (i), each $k_j \ne 2$. By analysing the proofs of Claims \ref{cl:preserves characteristic polynomial}, \ref{cl:identity on diagonals}, \ref{cl:preserves diagonalizability} and \ref{cl:preserves zero-product} we observe that they do not require the injectivity of the map $\phi$. Therefore, without loss of generality, one can assume that $\phi$ acts as the identity map on $\ca{D}_n$. Further, using Claim \ref{cl:preserves zero-product} and a standard density argument, it easily follows that $\phi$ maps each diagonal block of $\ca{A}_\rho$ into itself. Therefore, for each $1 \le j \le k$ the map $\phi$ restricts to a continuous spectrum and commutativity preserver $\phi_j : M_{k_j} \to M_{k_j}$. If $k_j = 1$, clearly $\phi_j$ is the identity, while if $k_j \ge 3$ we apply Theorem \ref{thm:Semrl} to $\phi_j$ to conclude that it acts as the identity or as the transposition map. Putting everything back together, it follows that $\phi$ is a Jordan embedding. Now, on the other hand, suppose  that $\ca{A}_\rho$ is not semisimple. Again, in view of Remark \ref{re:invariant under similarity} and  \cite[p.\ 432]{Akkurt}, we can assume that $\ca{A}_\rho$ is in the block upper-triangular form with at least one nonzero entry in the strict block upper triangle. Then, as in \cite[Example 5.4]{GogicPetekTomasevic}, a map $\phi : \ca{A}_\rho \to M_n$ which only leaves the diagonal blocks intact and annihilates everything else is an example of a non-injective unital Jordan homomorphism (in particular, $\phi$ is a continuous spectrum and commutativity preserver).
		\end{itemize}
	\end{remark}

\end{document}